\documentclass[10pt,reqno]{amsart}
\usepackage{amsmath,amssymb,latexsym,esint,cite,mathrsfs}
\usepackage{verbatim,wasysym,cite,relsize,color,soul}
\usepackage[latin1]{inputenc}
\usepackage{microtype}
\usepackage{color,enumitem,graphicx}
\usepackage[colorlinks=true,urlcolor=blue, citecolor=red,linkcolor=blue,
linktocpage,pdfpagelabels, bookmarksnumbered,bookmarksopen]{hyperref}
\usepackage[hyperpageref]{backref}
\usepackage[english]{babel}
\usepackage{tikz}
\usepackage[font=small,skip=5pt]{caption}
\usepackage{amsrefs}

\usepackage{geometry}
\numberwithin{equation}{section}
\newtheorem{theorem}{Theorem}[section]
\newtheorem{lemma}[theorem]{Lemma}

\newtheorem{proposition}[theorem]{Proposition}
\newtheorem{corollary}[theorem]{Corollary}

\newtheoremstyle{remarkstyle}
{\topsep}{\topsep}{ }{ }{\bfseries}{.~}{ }{\thmname{#1}\thmnumber{ #2}\thmnote{ (#3)}}
\theoremstyle{remarkstyle}
\newtheorem{remark}{Remark}[section]
\newtheorem{definition}{Definition}[section]

\newcommand{\R}{\mathbb R}

\newcommand{\Fc}{\mathcal F}

\newcommand{\vareps}{\varepsilon}

\DeclareMathOperator*{\loc}{loc}
\DeclareMathOperator*{\rad}{rad}
\DeclareMathOperator*{\opt}{opt}

\DeclareMathOperator*{\gamc}{{\gamma_c}}
\DeclareMathOperator*{\sigc}{{\sigma_c}}

\DeclareMathOperator*{\qbo}{{\boldsymbol{q}}}
\DeclareMathOperator*{\rbo}{{\boldsymbol{r}}}
\DeclareMathOperator*{\kbo}{{\boldsymbol{k}}}
\DeclareMathOperator*{\mbo}{{\boldsymbol{m}}}
\DeclareMathOperator*{\abo}{{\boldsymbol{a}}}
\DeclareMathOperator*{\bbo}{{\boldsymbol{b}}}
\DeclareMathOperator*{\lbo}{{\boldsymbol{l}}}
\DeclareMathOperator*{\nbo}{{\boldsymbol{n}}}
\DeclareMathOperator*{\ima}{Im}
\DeclareMathOperator*{\rea}{Re}
\newcommand{\scal}[1]{\left\langle #1 \right\rangle}

\title[Dynamics for fourth-order NLS]
{Dynamics of radial solutions for the focusing fourth-order nonlinear Schr\"odinger equations}
\author[V. D. Dinh]{Van Duong Dinh}
\address[V. D. Dinh]{Laboratoire Paul Painlev\'e UMR 8524, Universit\'e de Lille CNRS, 59655 Villeneuve d'Ascq Cedex, France
and 
Department of Mathematics, HCMC University of Education, 280 An Duong Vuong, Ho Chi Minh, Vietnam}
\email{contact@duongdinh.com}


\begin{document}
	
	\begin{abstract}
	We consider the following class of focusing $L^2$-supercritical fourth-order nonlinear Schr\"odinger equations
	\[
	i\partial_t u - \Delta^2 u + \mu \Delta u = - |u|^\alpha u, \quad (t,x) \in \R \times \R^N,
	\]
	where $N\geq 2$, $\mu \geq 0$, and $\frac{8}{N}<\alpha<\alpha^*$ with $\alpha^*:=\frac{8}{N-4}$ if $N\geq 5$ and $\alpha^*=\infty$ if $N\leq 4$. By using the localized Morawetz estimates and radial Sobolev embedding, we establish the energy scattering below the ground state threshold for the equation with radially symmetric initial data. We also address the existence of finite time blow-up radial solutions to the equation. In particular, we show a sharp threshold for scattering and blow-up for the equation with radial data. Our scattering result not only extends the	one proved by Guo \cite{Guo}, where the scattering was proven for $\mu = 0$, but also provides an alternative simple proof that completely avoids the use of the concentration/compactness and rigidity argument. In the case $\mu > 0$, our blow-up result extends an earlier result proved by Boulenger-Lenzmann \cite{BL}, where the finite time blow-up was shown for initial data with negative energy.
	\end{abstract}

	\maketitle

	\section{Introduction}
	\label{S1}
	\setcounter{equation}{0}
	We are interested in the Cauchy problem for a class of the fourth-order nonlinear Schr\"odinger equations 
	\begin{equation} \label{4NLS}
		\left\{ 
		\begin{array}{rcl}
			i\partial_t u - \Delta^2 u + \mu \Delta u &=& \pm |u|^\alpha u, \quad (t,x) \in \R \times \R^N, \\
			u(0,x)&=& u_0(x),
		\end{array}
		\right.
	\end{equation}
	where $u: \mathbb{R} \times \mathbb{R}^N \rightarrow \mathbb{C}$, $u_0: \mathbb{R}^N \rightarrow \mathbb{C}$, $\mu \in \R$, and $\alpha>0$. The plus (resp. minus) sign in front of the nonlinearity corresponds to the defocusing (resp. focusing) case. The fourth-order Schr\"odinger equation has been introduced by Karpman \cite{Karpman} and Karpman-Shagalov \cite{KS} in order to take into consideration the role of small fourth-order dispersion terms in the propagation of intense laser beams in a bulk medium with Kerr nonlinearity. 
	
	The equation \eqref{4NLS} has formally the conservation of mass and energy
	\begin{align*}
	M(u(t))&=\int |u(t,x)|^2 dx = M(u_0), \tag{Mass} \\
	E_\mu(u(t)) &= \frac{1}{2} \int |\Delta u(t,x)|^2 dx +\frac{\mu}{2} \int |\nabla u(t,x)|^2 dx \pm \frac{1}{\alpha+2} \int |u(t,x)|^{\alpha+2} dx = E_\mu(u_0). \tag{Energy}
	\end{align*}
	In the case $\mu=0$, the equation \eqref{4NLS} enjoys the scaling invariance
	\[
	u_\lambda(t,x):= \lambda^{\frac{4}{\alpha}} u(\lambda^4 t, \lambda x), \quad \lambda>0.
	\]
	A direct computation shows that
	\[
	\|u_\lambda(0)\|_{\dot{H}^\gamma} = \lambda^{\gamma+\frac{4}{\alpha}-\frac{N}{2}} \|u_0\|_{\dot{H}^\gamma},
	\]
	where $\dot{H}^\gamma$ is the homogeneous Sobolev space of order $\gamma$. Thus, we define the critical exponent
	\begin{align} \label{defi-gamc}
	\gamc:= \frac{N}{2} -\frac{4}{\alpha}.
	\end{align}
	We also define the exponent
	\begin{align} \label{defi-sigc}
	\sigc:= \frac{2-\gamc}{\gamc} = \frac{8-(N-4)\alpha}{N\alpha-8}.
	\end{align}
	In view of the conservation laws above, the equation is said to be mass-critical (resp. mass and energy intercritical, and energy-critical) if $\gamc=0$ (resp. $0<\gamc<2$, and $\gamc=2$).
	
	In the last decade, the fourth-order Schr\"odinger equation has been attracting a lot of interest in mathematics, numerics and physics. Fibich-Ilan-Papanicolaou \cite{FIP} studied the existence of global $H^2$-solutions and gave some numerical observations showing the existence of finite time blow-up solutions. Artzi-Koch-Saut \cite{BaKS} established sharp dispersive estimates for the fourth-order Schr\"odinger operator. Pausader \cite{Pausader-DPDE, Pausader-DCDS, Pausader-JFA} and Miao-Xu-Zhao \cite{MXZ-1, MXZ-2} investigated the asymptotic behavior (or energy scattering) of global $H^2$-solutions in the energy-critical case. In the mass and energy intercritical case, the energy scattering for the defocusing problem was shown by Pausader \cite{Pausader-DPDE} in dimensions $N\geq 5$ and Pausader-Xia \cite{PX} in low dimensions (see also \cite{MZ}). In the mass-critical case, the asymptotic behavior of global $L^2$-solutions was proved by Pausader-Shao \cite{PS}. The asymptotic behavior of global solutions below the energy space was studied by Miao-Wu-Zhang \cite{MWZ} and the first author \cite{Dinh-NA}. In a seminal work \cite{BL} (see Theorem \ref{theo-blow-BL}), Boulenger-Lenzmann established the existence of finite time blow-up $H^2$-solutions. Dynamical properties such as mass-concentration and limiting profile of blow-up $H^2$-solutions were studied by Zhu-Yang-Zhang \cite{ZYZ-DPDE} and the first author \cite{Dinh-JDDE}. Dynamical properties of blow-up solutions below the energy space were studied in \cite{ZYZ-NA, Dinh-DPDE}. 
	
	Motivated by aforementioned results, we study the energy scattering below the ground state and the finite time blow-up of radial solutions to the focusing problem \eqref{4NLS}. Before stating our results, let us recall some known results related to \eqref{4NLS}. The local well-posedness for \eqref{4NLS} in the energy space $H^2$ was established in \cite{Pausader-DPDE} (see also. \cite{Dinh-BBMS}). It was claimed without proof (see \cite[Proposition 4.1]{Pausader-DPDE} that \eqref{4NLS} is locally well-posed in $H^2$ for $0<\alpha <\alpha^*$, where
	\begin{align} \label{defi-alpha-star}
	\alpha^*:= \left\{
	\renewcommand*{\arraystretch}{1.2}
	\begin{array}{cl}
	\frac{8}{N-4} &\text{if } N\geq 5, \\
	\infty &\text{if } N\leq 4.
	\end{array}
	\right.
	\end{align}
	The author in \cite{Pausader-DPDE} referred to \cite{Cazenave} for a similar proof of this result. Due to the appearance of biharmonic operator, the nonlinearity needs to have at least two derivatives in order to apply the argument of \cite{Cazenave}. One can use Strichartz estimates with a gain of derivatives (see \eqref{str-est-gain}) to lower the regularity requirement of nonlinearity. However, a careful consideration (see Remark \ref{rem-lwp}) shows that we only have the local well-posedness for \eqref{4NLS} in $H^2$ for 
	\begin{align} \label{cond-alpha}
	\left\{
	\renewcommand*{\arraystretch}{1.3}
	\begin{array}{ccl}
	\frac{2}{N} \leq \alpha <\alpha^* &\text{if}& N\geq 3, \\
	\alpha \geq 1 &\text{if} &N=1,2.
	\end{array}
	\right.
	\end{align}
	In the energy subcritical case, i.e. $\gamc<2$ or $0<\alpha <\alpha^*$, local solutions satisfy the following blow-up alternative: either $T^*=+\infty$ or $T^*<+\infty$ and 
	\[
	\lim_{t\nearrow T^*} \|\Delta u(t)\|_{L^2} =\infty,
	\]
	where $T^*$ is the maximal forward time of existence. 
	
	The existence of blow-up $H^2$-solutions to the focusing problem \eqref{4NLS} was recently established by Boulenger-Lenzmann \cite{BL}. This work gives rigorous mathematical proofs for previous numerical results in \cite{BFM, BF, BFM-SIAM}. More precisely, we have the following result.
	\begin{theorem}[\cite{BL}]  \label{theo-blow-BL}
		\begin{enumerate}
			\item (Mass-critical case) Let $N\geq 2$, $\mu \geq 0$, and $\alpha=\frac{8}{N}$. Let $u_0 \in H^2$ be radially symmetric satisfying $E_\mu(u_0)<0$. It holds that 
			\begin{itemize}
				\item if $\mu>0$, then the corresponding solution to the focusing problem \eqref{4NLS} blows up in finite time;
				\item if $\mu=0$, then the corresponding solution to the focusing problem \eqref{4NLS} either blows up in finite time or blows up in infinite time and satisfies
				\[
				\|\Delta u(t)\|_{L^2} \geq Ct^2, \quad \forall t\geq t_0
				\]
				with some constant $C= C(u_0)>0$ and $t_0=t_0(u_0)>0$. 
			\end{itemize}
			\item (Mass and enery intercritical case) Let $N\geq 2$, $\mu  \in \R$, $\frac{8}{N}<\alpha<\alpha^*$, and $\alpha \leq 8$. Let $u_0\in H^2$ be radially symmetric and satisfy one of the following conditions:
			\begin{itemize}
				\item If $\mu \ne 0$, we assume that
				\[
				\left\{
				\begin{array}{ll}
				E_\mu(u_0) <0 &\text{if } \mu>0 \\
				E_\mu(u_0) <-\kappa \mu^2 M(u_0) &\text{if } \mu<0
				\end{array}
				\right.
				\]
				with some constant $\kappa = \kappa(N,\alpha)>0$.
				\item If $\mu=0$, we assume either $E_0(u_0)<0$ or, if $E_0(u_0) \geq 0$, we suppose that
				\[
				E_0(u_0) [M(u_0)]^{\sigc} < E_0(Q) [M(Q)]^{\sigc}
				\]
				and
				\[
				\|\Delta u_0\|_{L^2} \|u_0\|^{\sigc}_{L^2}  > \|\Delta Q\|_{L^2} \|Q\|^{\sigc}_{L^2},
				\]
				where $Q$ is the ground state related to the elliptic equation
				\begin{align} \label{ell-equ-Q}
				\Delta^2 Q + Q - |Q|^\alpha Q=0.
				\end{align}
			\end{itemize}
		Then the corresponding solution to the focusing problem \eqref{4NLS} blows up in finite time.
		\item (Energy-critical case) Let $N\geq 5$, $\mu\in \R$, and $\alpha=\frac{8}{N-4}$. Let $u_0 \in H^2$ be radially symmetric and satisfy one of the following properties:
		\begin{itemize}
			\item If $\mu \ne 0$, we assume that
			\[
			\left\{
			\begin{array}{ll}
			E_\mu(u_0) <0 &\text{if } \mu>0 \\
			E_\mu(u_0) <-\kappa \mu^2 M(u_0) &\text{if } \mu<0
			\end{array}
			\right.
			\]
			with some constant $\kappa=\kappa(N)>0$. 
			\item If $\mu=0$, we assume that either $E_0(u_0)<0$ or, if $E_0(u_0)\geq 0$, we suppose that 
			\[
			E_0(u_0)<E_0(W)
			\]
			and 
			\[
			\|\Delta u_0\|_{L^2}>\|\Delta W\|_{L^2},
			\]
			where $W$ is the unique radial, non-negative solution to the elliptic equation
			\begin{align} \label{ell-equ-W}
			\Delta^2 W - |W|^{\frac{8}{N-4}} W=0.
			\end{align}
		\end{itemize}
		Then the corresponding solution to the focusing problem \eqref{4NLS} blows up in finite time.
		\end{enumerate}
	\end{theorem}

	Our first result is the following energy scattering below the ground state for the focusing problem \eqref{4NLS}.
	\begin{theorem} \label{theo-scat}
		Let $N\geq 2$, $\mu\geq 0$, and $\frac{8}{N}<\alpha<\alpha^*$. Let $u_0 \in H^2$ be radially symmetric and satisfy
		\begin{align} 
		E_\mu(u_0) [M(u_0)]^{\sigc} &< E_0(Q) [M(Q)]^{\sigc}, \label{cond-ener} \\
		\|\Delta u_0\|_{L^2} \|u_0\|^{\sigc}_{L^2} & < \|\Delta Q\|_{L^2} \|Q\|^{\sigc}_{L^2}. \label{cond-grad-gwp}
		\end{align}
		Then the corresponding solution to the focusing problem \eqref{4NLS} exists globally in time and scatters in $H^2$ in both directions, i.e. there exist $u_\pm \in H^2$ such that
		\[
		\lim_{t\rightarrow \pm \infty} \|u(t) - e^{-it(\Delta^2-\mu \Delta)} u_\pm\|_{H^2} =0.
		\]
	\end{theorem}
	
	\begin{remark}
		The condition $\mu\geq 0$ is due to global in time Strichartz estimates (see Section $\ref{S2}$) and the variational analysis (see Section $\ref{S3}$). When $\mu<0$, only local in time Strichartz estimates are available (see \cite{BaKS}), so it is not appropriate to discuss the energy scattering in this case. 
	\end{remark}
	
	\begin{remark}
		Theorem $\ref{theo-scat}$ extends the one proved by Guo \cite{Guo}, where the energy scattering below the ground state for \eqref{4NLS} with $\mu=0$ was studied by using the concentration/compactness and rigidity argument of Kenig-Merle \cite{KM}. We remark that the proof in \cite{Guo} relies on the following inhomogeneous Strichartz estimates (see \cite[Proposition 2.2]{Guo}) which do not seem clear to us
		\begin{align} \nonumber
		\left\|\int_0^t e^{i(t-s)\Delta^2} |u(s)|^\alpha u(s) ds \right\|_{L^q(I, L^r)} \lesssim \||u|^\alpha u\|_{L^{\frac{q}{\alpha+1}}(I,L^{\frac{r}{\alpha+1}})},
		\end{align}
		where
		\[
		(q,r)= \left(\frac{(N+4)\alpha}{4},\frac{(N+4)\alpha}{4}\right), \quad (q,r)= \left(2\alpha,\frac{N\alpha}{2}\right).
		\]
		In fact, according to the best known inhomogeneous Strichartz estimates for Schr\"odinger-type equations (including the biharmonic NLS), which were proved independently by Foschi \cite{Foschi} and Vilela \cite{Vilela}, we need to check the following conditions:
		\begin{align} 
		\frac{1}{q}+\frac{N}{r}<\frac{N}{2}, \quad \frac{1}{m}+\frac{N}{n} <\frac{N}{2} \label{cond-1}\\
		\frac{4}{q}+\frac{N}{r}+\frac{4}{m}+\frac{N}{n} = N \nonumber 
		\end{align}
		and
		\begin{align}
		\frac{N-4}{N} \leq \frac{r}{n} \leq \frac{N}{N-4}, \label{cond-3}
		\end{align}
		where $(m,n)$ is the dual pair of $\left(\frac{q}{\alpha+1},\frac{r}{\alpha+1} \right)$. It is easy to see that \eqref{cond-1} and \eqref{cond-3} are not satisfied for all $\frac{8}{N}<\alpha<\alpha^*$. Therefore, the result stated in \cite{Guo} is doubtful.
	\end{remark}

	Theorem \ref{theo-scat} extends the energy scattering for the classical NLS obtained in \cite{HR} to the biharmonic NLS. The proof of Theorem \ref{theo-scat} is based on recent arguments of Dodson-Murphy \cite{DM} and Dinh-Keraani \cite{DK} using localized Morawetz estimates and radial Sobolev embedding. It gives an alternative simple proof for the energy scattering that completely avoids the use of the concentration/compactness and rigidity argument. 
	
	Let us briefly describe the strategy of the proof of Theorem \ref{theo-scat}. It is divided into three main steps as follows. 
	
	\noindent {\bf Step 1. Scattering criteria.} By using Strichartz estimates and the standard contraction mapping argument, we show that if $u$ is a global solution to the focusing problem \eqref{4NLS} satisfying
	\[
	\|u\|_{L^\infty(\R, H^2)} \leq A
	\]
	for some constant $A>0$, then there exists $\delta =\delta(A)>0$ such that if 
	\begin{align} \label{scat-crite-intro}
	\|e^{-i(t-T)(\Delta^2-\mu\Delta)} u(T)\|_{L^{\kbo}([T,\infty), L^{\rbo})} <\delta
	\end{align}
	for some $T>0$, where
	\[
	\kbo:=\frac{4\alpha(\alpha+2)}{8-(N-4)\alpha}, \quad \rbo:=\alpha+2,
	\]
	then the solution scatters in $H^2$ forward in time.
	
	\noindent {\bf Step 2. Localized Morawetz estimates.}
	By using some variational analysis, we prove that under the assumptions \eqref{cond-ener} and \eqref{cond-grad-gwp}, the corresponding solution to the focusing problem \eqref{4NLS} exists globally in time, and there exist $\nu=\nu(u_0,Q)>0$ and $R_0=R_0(u_0,Q)>0$ such that for any $R\geq R_0$,
	\begin{align} \label{coer-intro}
	K_0(\chi_R(u(t)) \geq \nu \|\chi_R u(t)\|^{\alpha+2}_{L^{\alpha+2}}
	\end{align}
	for all $t \in \R$. Here 
	\[
	K_0(u):= \|\Delta u\|^2_{L^2} -\frac{N\alpha}{4(\alpha+2)} \|u\|^{\alpha+2}_{L^{\alpha+2}}
	\]
	is the virial functional and $\chi_R(x) = \chi(x/R)$ with $\chi \in C^\infty_0(\R^N)$ satisfying $0\leq \chi \leq 1$ and
	\[
	\chi(x) = \left\{
	\begin{array}{ccl}
	1 &\text{if}& |x| \leq 1/2, \\
	0 &\text{if}& |x| \geq 1.
	\end{array}
	\right.
	\]
	Thanks to the coercivity property \eqref{coer-intro}, localized Morawetz estimates, and the radial Sobolev embedding, we show that for any time interval $I\subset \R$,
	\begin{align} \label{space-time-est}
	\int_I \|u(t)\|^{\alpha+2}_{L^{\alpha+2}} dt \lesssim |I|^{\frac{1}{3}}.
	\end{align}
	
	\noindent {\bf Step 3. Energy scattering.}
	By Step 1, it suffices to find $T>0$ so that \eqref{scat-crite-intro} holds. To reach this goal, let $\vareps>0$ be a small parameter. For $T>\vareps^{-\sigma}$ with some $\sigma>0$ to be chosen later, we write
	\[
	e^{-i(t-T)(\Delta^2-\mu\Delta)} u(T) = e^{-it (\Delta^2-\mu\Delta)} u_0 + F_1(t) + F_2(t),
	\]
	where
	\[
	F_1(t):=i \int_I e^{-i(t-s)(\Delta^2-\mu\Delta)} |u(s)|^\alpha u(s) ds, \quad F_2(t):= i\int_J e^{-i(t-s)(\Delta^2-\mu\Delta)} |u(s)|^\alpha u(s) ds
	\]
	with $I:= [T-\vareps^{-\sigma}, T]$ and $J:= [0,T-\vareps^{-\sigma}]$. The smallness of the linear part follows easily from Strichartz estimates by taking $T>\vareps^{-\sigma}$ sufficiently large. The smallness of $F_1$ follows from Strichartz estimates, \eqref{space-time-est} and the radial Sobolev embedding. Finally, the smallness of $F_2$ is based on dispersive estimates and \eqref{space-time-est}. We refer the reader to Section $\ref{S4}$ for more details.

	Our next result concerns the finite time blow-up in the mass and energy intercritical case.
	\begin{theorem} \label{theo-blup-inter}
		Let $N\geq 2$, $\mu\geq 0$, $\frac{8}{N}<\alpha<\alpha^*$, and $\alpha \leq 8$. Let $u_0 \in H^2$ be radially symmetric satisfying \eqref{cond-ener} and
		\begin{align} \label{cond-grad-blup}
		\|\Delta u_0\|_{L^2} \|u_0\|^{\sigc}_{L^2} > \|\Delta Q\|_{L^2} \|Q\|^{\sigc}_{L^2}.
		\end{align}
		Then the corresponding solution to the focusing problem \eqref{4NLS} blows up in finite time.
	\end{theorem}
	
	\begin{remark}
		The restriction $\alpha \leq 8$ is technical due to the radial Sobolev embedding (see Lemma $\ref{lem-est-mora-rad}$).
	\end{remark}
	
	\begin{remark}
		In the case $\mu>0$, Theorem $\ref{theo-blup-inter}$ extends the result in \cite{BL}, where the finite time blow-up for radial initial data with negative energy was shown.
	\end{remark}

	\begin{remark}
		In \cite{BCGJ}, a finite time blow-up result for radial non-negative energy $H^2$-solutions for \eqref{4NLS} with $\mu>0$ was shown. However, this result is not directly applicable to Theorem \ref{theo-blup-inter}.
	\end{remark}
	
	\begin{remark}
		We will see from Remark \ref{rem-data} that there is no $u_0 \in H^2$ satisfies \eqref{cond-ener} and 
		\[
		\|\Delta u_0\|_{L^2} \|u_0\|^{\sigc}_{L^2} = \|\Delta Q\|_{L^2} \|Q\|^{\sigc}_{L^2}.
		\] 
		Thus, Theorems \ref{theo-scat} and \ref{theo-blup-inter} give a sharp threshold for the scattering and finite time blow-up for \eqref{4NLS}.
	\end{remark}

	The proof of Theorem \ref{theo-blup-inter} is based on a variational analysis and an ODE argument of Boulenger-Lenzmann \cite{BL}. We first show that under the assumptions \eqref{cond-ener} and \eqref{cond-grad-blup}, there exists $\delta =\delta(u_0,Q)>0$ such that the corresponding solution to the focusing problem \eqref{4NLS} satisfies
	\[
	K_\mu(u(t)) \leq -\delta
	\] 
	for all $t$ in the existence time, where
	\[
	K_\mu(u):= \|\Delta u\|^2_{L^2} +\frac{\mu}{2} \|\nabla u\|^2_{L^2} -\frac{N\alpha}{4(\alpha+2)} \|u\|^{\alpha+2}_{L^{\alpha+2}}.
	\]
	Thanks to the above bound and localized Morawetz estimates, we show that there exists $a=a(u_0,Q)>0$ such that
	\[
	\frac{d}{dt} M_{\varphi_R}(t) \leq -a \|\Delta u(t)\|^2_{L^2}
	\]
	for all $t$ in the existence time. With this bound at hand, an ODE argument of \cite{BL} shows that the solution must blow up in finite time. We refer the reader to Section $\ref{S5}$ for more details.

	Finally, we have the following finite time blow-up in the energy critical case.
	\begin{theorem} \label{theo-blup-ener}
		Let $N\geq 5$, $\mu\geq 0$, and $\alpha=\frac{8}{N-4}$. Let $u_0 \in H^2$ be radially symmetric satisfying
		\begin{align} 
		E_\mu(u_0) &< E_0(W), \label{cond-ener-ener} \\
		\|\Delta u_0\|_{L^2} &> \|\Delta W\|_{L^2}, \label{cond-grad-ener}
		\end{align}
		where $W$ is the unique non-negative radial solution to \eqref{ell-equ-W}. Then the corresponding solution to the focusing problem \eqref{4NLS} blows up in finite time.
	\end{theorem}
	
	The proof of this result follows the same argument as in the proof of Theorem $\ref{theo-blup-inter}$ using \eqref{cond-ener-ener} and \eqref{cond-grad-ener}.
	
	\begin{remark}
	In the case $\mu>0$, this result extends the one in \cite{BL}, where the finite time blow-up for radial initial data with negative energy was shown.	
	\end{remark}
	
	This paper is organized as follows. In Section $\ref{S2}$, we give some preliminaries including dispersive and Strichartz estimates. In Section $\ref{S3}$, we prove the local well-posedness for \eqref{4NLS}. The proof of the energy scattering below the ground state is given in Section $\ref{S4}$. Finally, the finite time blow-up given Theorem $\ref{theo-blup-inter}$ and Theorem $\ref{theo-blup-ener}$ will be proved in Section $\ref{S5}$. 
	
	\section{Strichartz estimates}
	\label{S2}
	\setcounter{equation}{0}
	Let $\mu \in \R$ and $e^{-it(\Delta^2 -\mu \Delta)}$ be the propagator for the free fourth-order Schr\"odinger equation 
	\[
	i\partial_t u -\Delta^2 u +\mu \Delta u =0. 
	\]
	The Schr\"odinger operator is defined by
	\[
	e^{-it(\Delta^2 -\mu \Delta)} f = \Fc^{-1}[e^{-it(|\xi|^4 +\mu|\xi|^2)} \Fc (f)],
	\]
	where $\Fc$ and $\Fc^{-1}$ are the Fourier and inverse Fourier transforms given by
	\[
	\Fc(f)(\xi) := (2\pi)^{-\frac{N}{2}} \int_{\R^N} e^{-ix\cdot\xi} f(x) dx, \quad \Fc^{-1}(g)(x) := (2\pi)^{-\frac{N}{2}} \int_{\R^N} e^{ix\cdot \xi} g(\xi) d\xi.
	\]
	Let $I_\mu$ be the distributional kernel of $e^{-it(\Delta^2-\mu \Delta)}$, i.e. 
	\[
	e^{-it(\Delta^2-\mu\Delta)} f(x) = I_\mu(t,x) \ast f(x),
	\]
	where $\ast$ is the convolution operator. We see that
	\[
	I_\mu(t,x):= (2\pi)^{-N} \int_{\R^N} e^{-it(|\xi|^4+\mu|\xi|^2) - i x \cdot \xi} d\xi.
	\] 
	Note that $I_\mu(t,x) = J_{-\mu}(-t, x)$, where
	\[
	J_\mu(t,x):= (2\pi)^{-N} \int_{\R^N} e^{it(|\xi|^4-\mu|\xi|^2)- ix\cdot \xi} d\xi.
	\]
	Dispersion estimates for $J_\mu(t)$ have been studied by Ben-Artzi-Koch-Saut \cite{BaKS}. More precisely, the following estimates hold true:
	\begin{itemize}
		\item ($\mu=0$)
		\begin{align} \label{dis-mu0}
		|D^\beta J_0(t,x)| \leq C t^{-\frac{N+|\beta|}{4}} \left(1+t^{-\frac{1}{4}} |x|\right)^{\frac{|\beta|-N}{3}}
		\end{align}
		for all $t>0$ and all $x \in \R^N$.
		\item ($\mu\in \{0,\pm 1\}$)
		\begin{align} \label{dis-mu}
		|D^\beta J_\mu(t,x)| \leq C t^{-\frac{N+|\beta|}{4}} \left(1+t^{-\frac{1}{4}}|x|\right)^{\frac{|\beta|-N}{3}}
		\end{align}
		for all $0<t \leq 1$ and all $x \in \R^N$, or all $t>0$ and all $|x| \geq t$.
		\item ($\mu=-1$)
		\begin{align} \label{dis-mu-1}
		|D^\beta J_{-1}(t,x)| \leq C t^{-\frac{N+|\beta|}{2}} \left(1+t^{-\frac{1}{2}} |x|\right)^{|\beta|}
		\end{align}
		for all $t\geq 1$ and all $|x| \leq t$.
	\end{itemize}
	Here $D$ stands for the differentiation in the $x$ variable. Useful consequences of \eqref{dis-mu0}, \eqref{dis-mu} and \eqref{dis-mu-1} are the followings:
	\begin{align*} 
	|J_\mu(t,x)| \leq C |t|^{-\frac{N}{4}}
	\end{align*}
	for all $t\ne 0$ and if $\mu=1$, we require $|t| \leq 1$. Note that $J_\mu(-t,x) = \overline{J_\mu(t,x)}$. It follows that
	\begin{align} \nonumber
	|I_\mu(t,x)| = |J_{-\mu}(-t,x)| \leq C |t|^{-\frac{N}{4}}
	\end{align}
	for all $t\ne 0$ and if $\mu=-1$, we require $|t| \leq 1$. This implies that
	\begin{align*}
	\|e^{-it(\Delta^2-\mu\Delta)} f\|_{L^\infty} \leq \|I_\mu(t)\|_{L^\infty} \|f\|_{L^1} \leq C |t|^{-\frac{N}{4}} \|f\|_{L^1}
	\end{align*}
	which together with the Riesz-Thorin interpolation theorem imply
	\begin{align}  \label{disper-est}
	\|e^{-it(\Delta^2-\mu\Delta)} f\|_{L^r} \leq C|t|^{-\frac{N}{4}\left(1-\frac{2}{r}\right)} \|f\|_{L^{r'}}
	\end{align}
	for all $r\in [2,\infty]$, all $f\in L^{r'}$, where $r'$ is the conjugate exponent of $r$ and all $t\ne 0$ and if $\mu=-1$, we require $|t| \leq 1$. Since we are interested in the energy scattering for \eqref{4NLS}, we only consider $\mu\geq 0$ throughout this paper.

	Let $I\subset \R$ and $q,r \in [1,\infty]$. We define the mixed norm 
	\[
	\|u\|_{L^q(I,L^r)} := \left( \int_I \left( \int_{\R^N} |u(t,x)|^r dx \right)^{\frac{q}{r}} dt \right)^{\frac{1}{q}}
	\]
	with a usual modification when either $q$ or $r$ are infinity. When $q=r$, we use the notation $L^q(I \times \R^N)$ instead of $L^q(I,L^q)$.
	\begin{definition}
		A pair $(q,r)$ is said to be {\bf Biharmonic admissible}, or $(q,r)\in B$ for short, if
		\[
		\frac{4}{q}+\frac{N}{r} =\frac{N}{2}, \quad \left\{
		\renewcommand*{\arraystretch}{1.3}
		\begin{array}{ll}
		r \in \left[2,\frac{2N}{N-4}\right] &\text{if } N\geq 5, \\
		r \in [2,\infty) &\text{if } N=4, \\
		r \in [2,\infty] &\text{if } N\leq 3.
		\end{array}
		\right.
		\]
		A pair $(m,n)$ is said to be {\bf Schr\"odinger admissible}, or $(m,n)\in S$ for short, if
		\[
		\frac{2}{m}+\frac{N}{n} =\frac{N}{2}, \quad \left\{
		\renewcommand*{\arraystretch}{1.3}
		\begin{array}{ll}
		n \in \left[2,\frac{2N}{N-2}\right] &\text{if } N\geq 3, \\
		n \in [2,\infty) &\text{if } N=2, \\
		n \in [2,\infty] &\text{if } N=1.
		\end{array}
		\right.
		\]
	\end{definition}
	Let $I\subset \R$ be an interval. We denote the Strichartz norm and its dual norm respectively by
	\[
	\|u\|_{S(I,L^2)}:= \sup_{(q,r) \in B} \|u\|_{L^q(I,L^r)}, \quad \|u\|_{S'(I,L^2)} := \inf_{(q,r)\in B} \|u\|_{L^{q'}(I,L^{r'})}.
	\]
	
	Thanks to dispersive estimates \eqref{disper-est} and the abstract theory of Keel-Tao \cite{KT}, we have the following Strichartz estimates.
	
	\begin{proposition} [Strichartz estimates \cite{Pausader-DPDE,Dinh-DPDE}] \label{prop-str-est}
		Let $\mu \geq 0$ and $I \subset \R$ be an interval. Then there exists a constant $C>0$ independent of $I$ such that the following estimates hold true.
		\begin{itemize}
			\item (Homogeneous estimates)
			\begin{align}
			\|e^{-it(\Delta^2-\mu \Delta)} f\|_{S(I,L^2)} \leq C \|f\|_{L^2}. \label{homo-str-est}
			\end{align}
			\item (Inhomogeneous estimates)
			\begin{align}
			\left\| \int_0^t e^{-i(t-s)(\Delta^2-\mu\Delta)} F(s)ds \right\|_{S(I,L^2)} &\leq C\|F\|_{S'(I,L^2)}. \label{inho-str-est-1}
			\end{align}
	\end{itemize}
	\end{proposition}

	We also have the following Strichartz estimates with a gain of derivatives (see e.g. \cite[Proposition 3.2]{Pausader-DPDE} or \cite{Dinh-DPDE}). 
	\begin{lemma} [Strichartz estimates with a gain of derivatives \cite{Pausader-DPDE,Dinh-DPDE}]
		Let $\mu\geq 0$ and $I \subset \R$ be an interval. Then there exists a constant $C>0$ independent of $I$ such that
		\begin{align} \label{str-est-gain}
		\left\| \Delta \int_0^t e^{-i(t-s)(\Delta^2-\mu\Delta)} F(s)ds \right\|_{L^q(I,L^r)} \leq C \||\nabla|^{2-\frac{2}{m}} F\|_{L^{m'}(I,L^{n'})} 
		\end{align}
		for any $(q,r)\in B$ and any $(m,n) \in S$. In particular, we have for $N\geq 3$,
		\begin{align}
		\left\| \Delta \int_0^t e^{-i(t-s)(\Delta^2-\mu \Delta)} F(s) ds \right\|_{S(I,L^2)} &\leq C\|\nabla F\|_{L^2(I,L^{\frac{2N}{N+2}})}. \label{inho-str-est-2}
		\end{align}
	\end{lemma}
	
	We also have the following Strichartz estimates for non-admissible pairs.
	\begin{lemma} \label{lem-str-est-non-adm}
		Let $\mu\geq 0$ and $I \subset \R$ be an interval. Let $(q,r)$ be a Biharmonic admissible pair with $r>2$. Fix $k>\frac{q}{2}$ and define $m$ by
		\begin{align} \label{cond-km}
		\frac{1}{k} +\frac{1}{m} =\frac{2}{q}.
		\end{align}
		Then there exists $C=C(N,q,r,k,m)>0$ such that
		\begin{align} \label{str-est-non-adm}
		\left\| \int_0^t e^{-i(t-s)(\Delta^2-\mu \Delta)} F(s) ds \right\|_{L^k(I,L^r)} \leq C\|F\|_{L^{m'}(I,L^{r'})}
		\end{align}
		for any $F \in L^{m'}(I,L^{r'})$.
	\end{lemma}
	
	\begin{proof}
	Thanks to \eqref{disper-est}, we have
	\[
	\left\| \int_0^t e^{-i(t-s)(\Delta^2-\mu \Delta)} F(s) ds \right\|_{L^r} \lesssim \int_0^t |t-s|^{-\frac{N}{4}\left(1-\frac{2}{r}\right)} \|F(s)\|_{L^{r'}} ds = \int_0^t |t-s|^{-\frac{2}{q}} \|F(s)\|_{L^{r'}} ds.
	\]
	The result follows from the Hardy-Littlewood-Sobolev inequality and \eqref{cond-km}.
	\end{proof}

\section{Local theory}
\label{S3}
\setcounter{equation}{0}
In this section, we prove the local well-posedness in $H^2$ and the small data theory for \eqref{4NLS}. Let us start with the following nonlinear estimates.
	\begin{lemma} \label{lem-non-est-lwp-1}
		Let $N\geq 1$, $0<\alpha<\alpha^*$ and $I \subset \R$ be an interval. Then there exists $\theta>0$ such that 
		\begin{align}
		\||u|^\alpha u\|_{S'(I,L^2)} \lesssim |I|^\theta \|(1-\Delta) u\|^\alpha_{S(I,L^2)} \|u\|_{S(I,L^2)}. \label{est-lwp-1}
		\end{align}
	\end{lemma}
	
	\begin{proof}
		We consider separately two cases: $N\geq 5$ and $N\leq 4$. 
		
		$\bullet$ When $N\geq 5$, we introduce
		\[
		(q,r):= \left( \frac{8(\alpha+2)}{(N-4)\alpha}, \frac{N(\alpha+2)}{N+2\alpha} \right), \quad (a,b):= \left( \frac{4\alpha(\alpha+2)}{8-(N-8)\alpha}, \frac{N(\alpha+2)}{N-4} \right).
		\]
		We readily check that $(q,r) \in B$,
		\[
		\frac{1}{q'}=\frac{\alpha}{a}+\frac{1}{q}, \quad \frac{1}{r'}=\frac{\alpha}{b}+\frac{1}{r}, \quad \frac{\alpha}{a} -\frac{\alpha}{q}=1-\frac{(N-4)\alpha}{8}
		\]
		and $\dot{W}^{2,r} \subset L^b$. By H\"older's inequality and Sobolev embedding, we have
		\begin{align*}
		\||u|^\alpha u\|_{S'(I,L^2)} \leq \||u|^\alpha u\|_{L^{q'}(I, L^{r'})} &\leq \|u\|^\alpha_{L^a(I,L^b)} \|u\|_{L^q(I,L^r)} \\
		&\lesssim |I|^{1-\frac{(N-4)\alpha}{8}} \|\Delta u\|^\alpha_{L^q(I,L^r)} \|u\|_{L^q(I,L^r)}.
		\end{align*}
		
		$\bullet$ When $N\leq 4$, we take the advantage of the Sobolev embedding $H^2 \subset L^r$ for all $r\in [2,\infty)$. We introduce
		\[
		(q,r)= \left(\frac{8(\alpha+1)}{N\alpha}, 2(\alpha+1)\right), \quad (a,b)=\left(\frac{8\alpha(\alpha+1)}{8-(N-8)\alpha}, 2(\alpha+1)\right), \quad (m,n)=(\infty,2).
		\]
		We see that $(q,r), (m,n) \in B$,
		\[
		\frac{1}{q'}=\frac{\alpha}{a}+\frac{1}{m}, \quad \frac{1}{r'} =\frac{\alpha}{b}+\frac{1}{n}
		\]
		and $H^2 \subset L^b$. By H\"older's inequality and Sobolev embedding, we have
		\begin{align}
		\||u|^\alpha u\|_{S'(I,L^2)} \leq \||u|^\alpha u\|_{L^{q'}(I, L^{r'})} &\leq \|u\|^\alpha_{L^a(I,L^b)} \|u\|_{L^m(I,L^n)} \nonumber\\
		&\lesssim |I|^{1-\frac{N\alpha}{8(\alpha+1)}} \|(1-\Delta) u\|^\alpha_{L^m(I,L^n)} \|u\|_{L^m(I,L^n)}. \label{est-ab-qr}
		\end{align}
		Collecting the above two cases, we get \eqref{est-lwp-1}.
	\end{proof}
	
	\begin{lemma} \label{lem-non-est-lwp-2}
		Let $N\geq 3$, $\frac{2}{N} \leq \alpha <\alpha^*$ and $I \subset \R$ be an interval.  Then there exists $\theta>0$ such that
		\begin{align}
		\|\nabla (|u|^\alpha u)\|_{L^2(I,L^{\frac{2N}{N+2}})} \lesssim |I|^\theta \|(1-\Delta) u\|^{\alpha+1}_{S(I,L^2)}. \label{est-lwp-2}
		\end{align}
	\end{lemma}
	
	\begin{proof}
		We consider two cases: $N\geq 5$ and $3\leq N\leq 4$.
		
		$\bullet$ When $N\geq 5$, we consider two subcases. If $\frac{4}{N-4}<\alpha<\frac{8}{N-4}$, we introduce
		\[
		a=2(\alpha+1), \quad b=\frac{2N(\alpha+1)}{(N-4)\alpha+N}, \quad n=\frac{2N(\alpha+1)}{(N-2)\alpha+N+2}, \quad q=\frac{8(\alpha+1)}{(N-4)\alpha-4}, \quad r=\frac{2N(\alpha+1)}{N+4\alpha+4}.
		\]
		It is easy to check that $(q,r)\in B$ and 
		\[
		\frac{1}{2}=\frac{\alpha+1}{a}, \quad \frac{N+2}{2N}=\frac{\alpha}{b}+\frac{1}{n}, \quad \frac{1}{b}=\frac{1}{r}-\frac{2}{N},\quad  \frac{1}{n}=\frac{1}{r}-\frac{1}{N}.
		\]
		By H\"older's inequality, we see that
		\begin{align*}
		\|\nabla(|u|^\alpha u)\|_{L^2(I,L^{\frac{2N}{N+2}})} &\leq \|u\|^\alpha_{L^a(I,L^b)} \|\nabla u\|_{L^a(I,L^n)} \\
		&\lesssim \|\Delta u\|^{\alpha+1}_{L^a(I,L^r)} \\
		&\lesssim |I|^{1-\frac{(N-4)\alpha}{8}} \|\Delta u\|^{\alpha+1}_{L^q(I,L^r)}.
		\end{align*}
		
		If $\frac{2}{N} \leq \alpha \leq \frac{4}{N-4}$, we estimate 
		\begin{align*}
		\|\nabla(|u|^\alpha u)\|_{L^2(I,L^{\frac{2N}{N+2}})} &\leq \|u\|^\alpha_{L^{2\alpha}(I,L^b)} \|\nabla u\|_{L^\infty(I,L^n)} \\
		&\lesssim \|(1-\Delta) u\|^\alpha_{L^{2\alpha}(I,L^2)} \|(1-\Delta) u\|_{L^\infty(I,L^2)} \\
		&\lesssim |I|^{\frac{1}{2}} \|(1-\Delta) u\|^{\alpha+1}_{L^\infty(I,L^2)},
		\end{align*}
		where $b$ and $n$ are chosen so that 
		\begin{align} \label{defi-b-n}
		\frac{N+2}{2N} = \frac{\alpha}{b} +\frac{1}{n}
		\end{align}
		the embeddings $H^2 \subset L^b$ and $H^1 \subset L^n$ hold. The latter condition implies that $b \in \left[2,\frac{2N}{N-4}\right]$ and $n \in \left[2,\frac{2N}{N-2}\right]$. This shows that 
		\[
		\frac{\alpha(N-4)}{2N} + \frac{N-2}{2N} \leq \frac{\alpha}{b} +\frac{1}{n} \leq \frac{\alpha+1}{2}.
		\]
		Since $\frac{2}{N} \leq \alpha \leq \frac{4}{N-4}$, we can choose $b \in \left[2,\frac{2N}{N-4}\right]$ and $n \in \left[2,\frac{2N}{N-2}\right]$ so that \eqref{defi-b-n} is satisfied.
		
		$\bullet$ In the case $3\leq N\leq 4$, we make use of the Sobolev embedding $H^2 \subset L^r$ for all $r\in [2,\infty)$. As $N\alpha\geq 2$, we have
		\begin{align*}
		\|\nabla(|u|^\alpha u)\|_{L^2(I,L^{\frac{2N}{N+2}})} &\leq \|u\|^\alpha_{L^{2\alpha}(I,L^{N\alpha})} \|\nabla u\|_{L^\infty(I,L^2)} \\
		&\lesssim \|(1-\Delta) u\|^\alpha_{L^{2\alpha}(I,L^2)} \|(1-\Delta)u\|_{L^\infty(I,L^2)} \\
		&\lesssim |I|^{\frac{1}{2}} \|(1-\Delta) u\|^{\alpha+1}_{L^\infty(I,L^2)}.
		\end{align*}
		Collecting the above cases, we end the proof.			
	\end{proof}

	\begin{lemma} [Local well-posedness] \label{lem-lwp}
		Let $N\geq 1$, $\mu \geq 0$ and $\alpha$ satisfy \eqref{cond-alpha}
		Let $u_0 \in H^2$. Then there exist $T_*, T^* \in (0,\infty]$ and a unique solution to \eqref{4NLS} satisfying
		\[
		u \in C((-T_*, T^*), H^2) \cap L^q_{\loc}((-T_*,T^*), W^{2,r})
		\]
		for all $(q,r) \in B$. Moreover, for any compact interval $I \Subset (-T_*,T^*)$ and any $(q,r) \in B$ with $q\ne \infty$,
		\begin{align} \label{loc-est-I}
		\|(1-\Delta)u\|_{L^q(I,L^r)} \lesssim \scal{I}^{\frac{1}{q}}.
		\end{align}
	\end{lemma}
	
	\begin{remark} \label{rem-lwp}
		The local well-posedness of $H^2$-solutions for \eqref{4NLS} was stated in \cite[Proposition 4.1]{Pausader-DPDE} for $0<\alpha<\alpha^*$ without proof. The author in \cite{Pausader-DPDE} refered to \cite{Cazenave} for a similar proof. However, due to a higher-order (Biharmonic) operator, we need the nonlinearity to have at least second derivatives to apply the method in \cite{Cazenave} (see also \cite{Dinh-BBMS}). This requires $\alpha\geq 1$ (hence $N\leq 12$) to get a similar result as for the classical NLS. To improve this restriction, we use Strichartz estimates with a gain of derivatives \eqref{inho-str-est-2}. This leads to the restriction \eqref{cond-alpha} (see Lemma $\ref{lem-non-est-lwp-2}$). Note that this restriction is sharp for the local well-posedness in $H^2$. In fact, to lower the requirement of regularity for the nonlinearity, we need to use \eqref{inho-str-est-2} which is the best estimate with a highest gain of derivatives in dimensions $N\geq 3$ (see \eqref{str-est-gain}). By H\"older's inequality, we estimate
		\begin{align} \label{sharp-lwp}
		\|\nabla (|u|^\alpha u)\|_{L^2(I, L^{\frac{2N}{N+2}})} \leq \|u\|^\alpha_{L^a(I,L^b)} \|\nabla u\|_{L^c(I,L^d)}
		\end{align}
		for some $a,b,c,d \in [1,\infty]$ satisfying
		\[
		\frac{1}{2}=\frac{\alpha}{a}+ \frac{1}{c}, \quad \frac{N+2}{2N} = \frac{\alpha}{b} +\frac{1}{d}.
		\]
		To bound the right hand side of \eqref{sharp-lwp} by $|I|^\theta \|(1-\Delta) u\|_{S(I,L^2)}^{\alpha+1}$ for some $\theta>0$, a necessary condition is $b, d\geq 2$ which leads to $\alpha\geq \frac{2}{N}$. In dimensions $N=1,2$, we use \eqref{str-est-gain} and estimate
		\[
		\||\nabla|^{2-\frac{2}{m}} (|u|^\alpha u)\|_{L^{m'}(I,L^{n'})} \leq \|u\|^\alpha_{L^a(I,L^b)} \||\nabla|^{2-\frac{2}{m}} u\|_{L^c(I,L^d)}.
		\]
		We see that in dimension $N=2$, the best estimate with gain of derivatives is $m=2+, n=\infty-$, hence $n'=1+$. Since $b,d\geq 2$, we need
		\[
		\alpha \geq \frac{2}{1+} -1>1.
		\]
		Here for a finite number $a$, the notation $a+$ stands for a number $a+\vareps$ with $0<\vareps \ll 1$. Similarly, $\infty- =\frac{1}{\vareps}$ with $0<\vareps \ll 1$. In dimension $N=1$, we take $m=4, n=\infty$, hence $n'=1$ and $\alpha \geq 1$. This shows that we cannot make the nonlinear exponent strictly smaller than 1 in dimensions $N=1,2$.
	\end{remark}
	
	\noindent {\it Proof of Lemma $\ref{lem-lwp}$.}
	Consider
	\[
	X:= \left\{ C(I, H^2) \cap \bigcap_{(q,r) \in B} L^q(I, W^{2,r}) \ : \ \|(1-\Delta) u\|_{S(I,L^2)} \leq M \right\}
	\]
	equipped with the distance
	\[
	d(u,v) := \|u-v\|_{S(I,L^2)},
	\]
	where $I=[-T,T]$ with $M, T>0$ to be chosen later. We will show that the functional
	\[
	\Phi_{u_0}(u(t)) := e^{-it(\Delta^2-\mu\Delta)} u_0 + i\int_0^t e^{-i(t-s)(\Delta^2-\mu\Delta)} |u(s)|^\alpha u(s) ds
	\]
	is a contraction on $(X,d)$. We will consider separately two cases: $N\geq 3$ and $1\leq N\leq 2$.
	
	In the case $N\geq 3$, by Strichartz estimates, Lemma $\ref{lem-non-est-lwp-1}$ and Lemma $\ref{lem-non-est-lwp-2}$, there exists $\theta_1, \theta_2 >0$ such that
	\begin{align}
	\|(1-\Delta) \Phi_{u_0}(u)\|_{S(I,L^2)} &\sim \|\Phi_{u_0}(u)\|_{S(I,L^2)} + \|\Delta \Phi_{u_0}(u)\|_{S(I,L^2)}  \nonumber \\
	&\lesssim \|u_0\|_{L^2} + \||u|^\alpha u\|_{S'(L^2,I)} + \|\Delta u_0\|_{L^2} + \|\nabla (|u|^\alpha u) \|_{L^2(I,L^{\frac{2N}{N+2}})} \nonumber \\
	&\lesssim \|u_0\|_{H^2} + \||u|^\alpha u\|_{S'(L^2,I)} + \|\nabla (|u|^\alpha u) \|_{L^2(I,L^{\frac{2N}{N+2}})} \nonumber \\
	&\lesssim \|u_0\|_{H^2} + |I|^{\theta_1} \|(1-\Delta) u\|^\alpha_{S(I,L^2)} \|u\|_{S(I,L^2)} + |I|^{\theta_2} \|(1-\Delta) u\|^{\alpha+1}_{S(I,L^2)} \nonumber \\
	&\lesssim  \|u_0\|_{H^2} + \left(|I|^{\theta_1} + |I|^{\theta_2}\right) \|(1-\Delta) u\|^{\alpha+1}_{S(I,L^2)}. \label{cont-est-I}
	\end{align}
	
	In the case $1\leq N\leq 2$, we use Lemma $\ref{lem-non-est-lwp-1}$ and \eqref{est-ab-qr} to have
	\begin{align}
	\|(1-\Delta) \Phi_{u_0}(u)\|_{S(I,L^2)} &\lesssim \|u_0\|_{H^2} + \||u|^\alpha u\|_{S'(L^2,I)} + \|\Delta (|u|^\alpha u) \|_{L^{q'}(I,L^{r'})} \nonumber \\
	&\lesssim \|u_0\|_{H^2} + \||u|^\alpha u\|_{S'(L^2,I)} + \|u\|^\alpha_{L^a(I,L^b)} \|\Delta u\|_{L^m(I,L^n)} \nonumber \\
	&\lesssim \|u_0\|_{H^2} + |I|^{\theta_1} \|(1-\Delta) u\|^\alpha_{S(I,L^2)} \|u\|_{S(I,L^2)} + |I|^{\theta_2} \|(1-\Delta) u\|^{\alpha+1}_{S(I,L^2)} \nonumber \\
	&\lesssim  \|u_0\|_{H^2} + \left(|I|^{\theta_1} + |I|^{\theta_2}\right) \|(1-\Delta) u\|^{\alpha+1}_{S(I,L^2)}. \label{cont-est-II}
	\end{align}
	Here we have used the high-order derivative estimate due to \cite[Lemma A.3]{Kato}: if $\alpha \geq 1$, then for $1<q, q_2 <\infty$ and $1<q_1\leq \infty$ satisfying $\frac{1}{q}=\frac{\alpha}{q_1} + \frac{1}{q_2}$,
	\begin{align} \label{kato-est}
	\|\Delta (|u|^\alpha u)\|_{L^q} \lesssim \|u\|^\alpha_{L^{q_1}} \|\Delta u\|_{L^{q_2}}.
	\end{align}
	Moroever, 
	\begin{align*}
	d(\Phi_{u_0}(u), \Phi_{u_0}(v)) &\lesssim \||u|^\alpha u - |v|^\alpha v\|_{S'(L^2,I)} \\
	&\lesssim |I|^{\theta_1} \left(\|(1-\Delta) u\|^\alpha_{S(I,L^2)} + \|(1-\Delta) v\|^\alpha_{S(I,L^2)} \right) \|u-v\|_{S(I,L^2)}.
	\end{align*}
	This shows that there exists $C>0$ independent of $u_0$ and $T$ such that for any $u, v\in X$,
	\begin{align*}
	\|(1-\Delta) \Phi_{u_0}(u)\|_{S(I,L^2)} &\leq C \|u_0\|_{H^2} + C(T^{\theta_1} + T^{\theta_2}) M^{\alpha+1}, \\
	d(\Phi_{u_0}(u), \Phi_{u_0}(v)) &\leq C T^{\theta_1} M^\alpha d(u,v).
	\end{align*}
	By taking $M=2C\|u_0\|_{H^2}$ and choosing $T>0$ small enough such that
	\[
	C\left(T^{\theta_1}+T^{\theta_2}\right)M^\alpha \leq \frac{1}{2},
	\]
	we see that $\Phi_{u_0}$ is a contraction on $(X,d)$. This shows the existence and uniqueness of solution to \eqref{4NLS}. The estimate \eqref{loc-est-I} follows from \eqref{cont-est-I} and \eqref{cont-est-II} by dividing $I$ into a finite number of small intervals and applying the continuity argument. The proof is now complete.
	\hfill $\Box$

	Let us now introduce some exponents
	\begin{align} \label{defi-qrkm}
	\begin{aligned}
	\qbo &:=\frac{8(\alpha+2)}{N\alpha}, & \rbo &:= \alpha+2, & \kbo&:=\frac{4\alpha(\alpha+2)}{8-(N-4)\alpha}, \\ 
	\mbo &:=\frac{4\alpha(\alpha+2)}{N\alpha^2+(N-4)\alpha-8}, & \abo &:=\frac{4(\alpha+2)}{(N-2)\alpha-4}, & \bbo&:=\frac{2N(\alpha+2)}{2(N+4)-(N-4)\alpha}.
	\end{aligned}
	\end{align}
	\begin{remark} \label{rem-qrkm}
	A straightforward computation shows that if $\frac{8}{N}<\alpha<\alpha^*$, then $(\qbo,\rbo)$ is a Biharmonic admissible pair. Moreover, the estimate \eqref{str-est-non-adm} holds for this choice of exponents since $\kbo, \mbo$ and $\qbo$ satisfy \eqref{cond-km}. 
	\end{remark}

	We also have the following nonlinear estimates which follow directly from H\"older's inequality, Sobolev embeddings, and \eqref{kato-est}.
	\begin{lemma} \label{lem-non-est}
		Let $N\geq 3$, $\frac{8}{N}<\alpha<\alpha^*$ and $I\subset \R$ be an interval. Then we have that
		\begin{align*}
		\||u|^\alpha u\|_{L^{\mbo'}(I,L^{\rbo'})} &\lesssim \|u\|^{\alpha+1}_{L^{\kbo}(I,L^{\rbo})}, \\
		\|\nabla(|u|^\alpha u)\|_{L^2(I,L^{\frac{2N}{N+2}})} &\lesssim \|u\|^\alpha_{L^{\kbo}(I,L^{\rbo})} \|\Delta u\|_{L^{\abo}(I,L^{\bbo})}.
		\end{align*}
		Moreover, if $\alpha \geq 1$, then 
		\begin{align*}
		\|\Delta(|u|^\alpha u)\|_{L^{\qbo'}(I,L^{\rbo'})} \lesssim \|u\|^\alpha_{L^{\kbo}(I, L^{\rbo})} \|\Delta u\|_{L^{\qbo}(I,L^{\rbo})}.
		\end{align*}
	\end{lemma}
	
	\begin{lemma} [Small data global well-posedness] \label{lem-small-gwp}
		Let $N\geq 1$, $\mu\geq 0$ and $\frac{8}{N}<\alpha<\alpha^*$. Let $T>0$ be such that 
		\[
		\|u(T)\|_{H^2} \leq A
		\]
		for some constant $A>0$. Then there exists $\delta=\delta(A)>0$ such that if 
		\[
		\|e^{-i(t-T)(\Delta^2-\mu \Delta)} u(T)\|_{L^{\kbo}([T,\infty), L^{\rbo})} \leq \delta,
		\]
		then the solution to \eqref{4NLS} with initial data $u(T)$ exists globally in time and satisfies
		\begin{align*}
		\|u\|_{L^{\kbo}([T,\infty),L^{\rbo})} &\leq 2\|e^{-i(t-T)(\Delta^2-\mu\Delta)} u(T)\|_{L^{\kbo}([T,\infty),L^{\rbo})}, \\
		\|u\|_{L^{\qbo}([T,\infty),W^{2,\rbo})} &\leq 2 C\|u(T)\|_{H^2},
		\end{align*}
		where $\qbo, \rbo, \kbo$ are as in \eqref{defi-qrkm}.
	\end{lemma}
	
	\begin{proof}
		We will consider separately two cases: $N\geq 5$ and $1\leq N\leq 4$. 
		
		{\bf Case 1. $N\geq 5$.} 
		We consider 
		\[
		Y:= \left\{ u \ : \ \|u\|_{L^{\kbo}(I,L^{\rbo})} \leq M, \quad \|u\|_{L^{\qbo}(I,W^{2,\rbo})} + \|u\|_{L^{\abo}(I,W^{2,\bbo})} \leq L \right\}
		\]
		equipped with the distance
		\[
		d(u,v):= \|u-v\|_{L^{\kbo}(I,L^{\rbo})} + \|u-v\|_{L^{\qbo}(I,L^{\rbo})} + \|u-v\|_{L^{\abo}(I,L^{\bbo})},
		\]
		where $I=[T,\infty)$, $M, L>0$ will be chosen later. Note that in this case, $(\qbo,\rbo)$ and $(\abo, \bbo)$ are Biharmonic admissible. We will show that the functional
		\[
		\Phi(u(t)) := e^{-i(t-T)(\Delta^2-\mu \Delta)} u(T) + i \int_T^t e^{-i(t-s)(\Delta^2-\mu \Delta)} |u(s)|^\alpha u(s) ds
		\]
		is a contraction on $(Y,d)$. Thanks to Remark $\ref{rem-qrkm}$, \eqref{str-est-non-adm} and Lemma $\ref{lem-non-est}$, we have
		\begin{align*}
		\|\Phi(u)\|_{L^{\kbo}(I,L^{\rbo})} &\leq \|e^{-i(t-T)(\Delta^2-\mu \Delta)} u(T)\|_{L^{\kbo}(I,L^{\rbo})} + \||u|^\alpha u\|_{L^{\mbo'}(I,L^{\rbo'})} \\
		&\leq \|e^{-i(t-T)(\Delta^2-\mu \Delta)} u(T)\|_{L^{\kbo}(I,L^{\rbo})} + \|u\|^{\alpha+1}_{L^{\kbo}(I,L^{\rbo})}.
		\end{align*}
		By Strichartz estimates and Lemma $\ref{lem-non-est}$, 
		\begin{align*}
		\|\Phi(u)\|_{L^{\qbo}(I, W^{2,\rbo})} &\sim \|\Phi(u)\|_{L^{\qbo}(I,L^{\rbo})} + \|\Delta \Phi(u)\|_{L^{\qbo}(I,L^{\rbo})} \\
		&\lesssim \|u(T)\|_{L^2} + \||u|^\alpha u\|_{L^{\qbo'}(I, L^{\rbo'})} + \|\Delta u(T)\|_{L^2} + \|\nabla(|u|^\alpha u)\|_{L^2(I,L^{\frac{2N}{N+2}})} \\
		&\lesssim \|u(T)\|_{H^2} + \|u\|^\alpha_{L^{\kbo}(I,L^{\rbo})} \|u\|_{L^{\qbo}(I,L^{\rbo})} + \|u\|^\alpha_{L^{\kbo}(I,L^{\rbo})} \|\Delta u\|_{L^{\abo}(I,L^{\bbo})} \\
		&\lesssim \|u(T)\|_{H^2} + \|u\|^\alpha_{L^{\kbo}(I,L^{\rbo})} \left( \|u\|_{L^{\qbo}(I,W^{2,\rbo})} + \|u\|_{L^{\abo}(I,W^{2,\bbo})} \right).
		\end{align*}
		Similarly,
		\begin{align*}
		\|\Phi(u)\|_{L^{\abo}(I,W^{2,\bbo})} &\sim \|\Phi(u)\|_{L^{\abo}(I,L^{\bbo})} + \|\Delta \Phi(u)\|_{L^{\abo}(I,L^{\bbo})} \\
		&\lesssim \|u(T)\|_{H^2} + \||u|^\alpha u\|_{L^{\qbo'}(I, L^{\rbo'})} + \|\Delta u(T)\|_{L^2} + \|\nabla(|u|^\alpha u)\|_{L^2(I,L^{\frac{2N}{N+2}})} \\
		&\lesssim \|u(T)\|_{H^2} + \|u\|^\alpha_{L^{\kbo}(I,L^{\rbo})} \left( \|u\|_{L^{\qbo}(I,W^{2,\rbo})} + \|u\|_{L^{\abo}(I,W^{2,\bbo})} \right).
		\end{align*}
		We also have
		\begin{align*}
		\|\Phi(u) - \Phi(v)\|_{L^{\kbo}(I,L^{\rbo})} &\lesssim \||u|^\alpha u - |v|^\alpha v\|_{L^{\mbo'}(I,L^{\rbo'})} \\
		&\lesssim \left(\|u\|^\alpha_{L^{\kbo}(I,L^{\rbo})} + \|v\|^\alpha_{L^{\kbo}(I,L^{\rbo})} \right) \|u-v\|_{L^{\kbo}(I,L^{\rbo})}
		\end{align*}
		and
		\begin{align*}
		\|\Phi(u) - \Phi(v)\|_{L^{\qbo}(I,L^{\rbo})} + \|\Phi(u) - \Phi(v)\|_{L^{\abo}(I,L^{\bbo})} &\lesssim \||u|^\alpha u - |v|^\alpha v\|_{L^{\qbo'}(I,L^{\rbo'})} \\
		&\lesssim \left(\|u\|^\alpha_{L^{\kbo}(I,L^{\rbo})} + \|v\|^\alpha_{L^{\kbo}(I,L^{\rbo})} \right) \|u-v\|_{L^{\qbo}(I,L^{\rbo})}.
		\end{align*}
		Thus, there exists $C>0$ independent of $T$ such that for any $u,v \in Y$,
		\begin{align*}
		\|\Phi(u)\|_{L^{\kbo}(I,L^{\rbo})} &\leq \| e^{-i(t-T)(\Delta^2-\mu \Delta)} u(T)\|_{L^{\kbo}(I,L^{\rbo})} + C M^{\alpha+1}, \\
		\|\Phi(u)\|_{L^{\qbo}(I,W^{2,\rbo})} +\|\Phi(u)\|_{L^{\abo}(I,W^{2,\bbo})} &\leq C\|u(T)\|_{H^2} + C M^\alpha L,
		\end{align*}
		and
		\[
		d(\Phi(u), \Phi(v)) \leq CM^\alpha d(u,v).
		\]
		By choosing $M = 2 \|e^{-i(t-T)(\Delta^2-\mu \Delta)} u(T)\|_{L^{\kbo}(I,L^{\rbo})}$, $L=2C\|u(T)\|_{H^2}$ and taking $M$ sufficiently small so that $CM^\alpha \leq \frac{1}{2}$, we see that $\Phi$ is a contraction on $(Y,d)$. 
		
		{\bf Case 2. $1\leq N\leq 4$.} In this case, since $\alpha>\frac{8}{N}$, we have $\alpha>1$. By the same argument as above, and using the following estimates
		\begin{align*}
		\|\Delta \Phi(u)\|_{L^{\qbo}(I,L^{\rbo})} &\lesssim \|\Delta u(T)\|_{L^2} + \|\Delta(|u|^\alpha u)\|_{L^{\qbo'}(I, L^{\rbo'})} \\
		&\lesssim \|\Delta u(T)\|_{L^2} + \|u\|^\alpha_{L^{\kbo}(I,L^{\rbo})} \|\Delta u\|_{L^{\qbo}(I,L^{\rbo})},
		\end{align*}
		we prove that $\Phi$ is a contraction on $(Y,d)$, where
		\[
		Y:= \left\{ u \ : \ \|u\|_{L^{\kbo}(I,L^{\rbo})} \leq M, \quad \|u\|_{L^{\qbo}(I,W^{2,\rbo})} \leq L \right\}
		\]
		and
		\[
		d(u,v):= \|u-v\|_{L^{\kbo}(I,L^{\rbo})} + \|u-v\|_{L^{\qbo}(I,L^{\rbo})}.
		\]
		
		Collecting the above cases, we complete the proof.
	\end{proof}

	\begin{lemma} [Small data scattering] \label{lem-small-scat}
		Let $N\geq 1$, $\mu\geq 0$ and $\frac{8}{N}<\alpha<\alpha^*$. Suppose that $u$ is a global solution to \eqref{4NLS} satisfying
		\[
		\|u\|_{L^\infty(\R, H^2)} \leq A
		\]
		for some constant $A>0$. Then there exists $\delta=\delta(A)>0$ such that if
		\[
		\|e^{-i(t-T)(\Delta^2-\mu \Delta)} u(T)\|_{L^{\kbo}([T,\infty),L^{\rbo})} <\delta
		\]
		for some $T>0$, then $u$ scatters in $H^2$ forward in time. 
	\end{lemma}
	
	\begin{proof}
		Let $\delta =\delta(A)$ be as in Lemma $\ref{lem-small-gwp}$. It follows from Lemma $\ref{lem-small-gwp}$ that the solution satisfies
		\begin{align*}
		\|u\|_{L^{\kbo}([T,\infty), L^{\rbo})} \leq 2 \|e^{-i(t-T)(\Delta^2-\mu\Delta)} u(T)\|_{L^{\kbo}([T,\infty),L^{\rbo})} \leq 2\delta
		\end{align*}
		and for $N\geq 5$,
		\[
		\|u\|_{L^{\qbo}([T,\infty),W^{2,\rbo})} + \|u\|_{L^{\abo}([T,\infty),W^{2,\bbo})} \leq 2C\|u(T)\|_{H^2} \leq 2CA
		\]
		and for $1\leq N\leq 4$,
		\[
		\|u\|_{L^{\qbo}([T,\infty),W^{2,\rbo})} \leq 2C\|u(T)\|_{H^2} \leq 2CA.
		\]
		Thanks to these global bounds, we show the energy scattering. For the reader's convenience, we give some details in the case $N\geq 5$. Let $0<\tau<t<\infty$. By Strichartz estimates, we see that 
		\begin{align*}
		\|e^{it(\Delta^2-\mu \Delta)} u(t) - e^{i\tau(\Delta^2-\mu\Delta)} u(\tau)\|_{H^2} &= \left\| \int_{\tau}^t e^{is(\Delta^2-\mu\Delta)} |u(s)|^\alpha u(s) ds \right\|_{H^2} \\
		&\lesssim \||u|^\alpha u\|_{L^{\qbo'}((\tau,t),L^{\rbo'})} + \|\nabla(|u|^\alpha u)\|_{L^2((\tau,t),L^{\frac{2N}{N+2}})} \\
		&\lesssim \|u\|^\alpha_{L^{\kbo}((\tau,t),L^{\rbo})} \left(\|u\|_{L^{\qbo}((\tau,t),L^{\rbo})} + \|\Delta u\|_{L^{\abo}((\tau,t),L^{\bbo})} \right) \\
		&\rightarrow 0
		\end{align*}
		as $\tau, t\rightarrow \infty$. This shows that $(e^{it(\Delta^2-\mu \Delta)} u(t))_t$ is a Cauchy sequence in $H^2$ as $t\rightarrow \infty$. Thus the limit
		\[
		u_+:= u_0 + i \int_0^\infty e^{is(\Delta^2-\mu\Delta)} |u(s)|^\alpha u(s) ds
		\]
		exists in $H^2$. Using the same argument as above, we prove that
		\[
		\|u(t)-e^{-it(\Delta^2-\mu\Delta)} u_+\|_{H^2} \rightarrow 0
		\]
		as $t\rightarrow \infty$. The proof is complete.
	\end{proof}

\section{Energy scattering}
\label{S4}
\setcounter{equation}{0}
In this section, we give the proof of the energy scattering for \eqref{4NLS} given in Theorem $\ref{theo-scat}$. 

\subsection{Variational analysis}
Let us recall some properties of the ground state $Q$ related to the elliptic equation \eqref{ell-equ-Q}. The ground state $Q$ optimizes the Gagliardo-Nirenberg inequality: $N\geq 1$, $0<\alpha<\alpha^*$,
\begin{align} \label{GN-ineq}
\|f\|^{\alpha+2}_{L^{\alpha+2}} \leq C_{\opt} \|\Delta f\|^{\frac{N\alpha}{4}}_{L^2} \|f\|^{\frac{8-(N-4)\alpha}{4}}_{L^2}, \quad f \in H^2(\R^N),
\end{align}
that is,
\[
C_{\opt}= \|Q\|^{\alpha+2}_{L^{\alpha+2}} \div \left[ \|\Delta Q\|^{\frac{N\alpha}{4}}_{L^2} \|Q\|^{\frac{8-(N-4)\alpha}{4}}_{L^2}\right].
\]
We note that the existence of ground states related to \eqref{ell-equ-Q} was proved by Zhu-Yang-Zhang \cite{ZYZ-DPDE}. Due to the presence of biharmonic operator, the classical argument using the symmetric rearrangment does not work. More precisely, the symmetric rearrangement of a $H^2$-function may not belong to $H^2$. To overcome the difficulty, they made use of the profile decomposition of bounded sequences in $H^2$ which was initially introduced by Hmidi-Keraani \cite{HK}. This profile decomposition can be seen as another description of the concentration-compactness principle of Lions \cite{Lions}. 

It was shown in \cite[Appendix]{BL} that $Q$ satisfies the following Pohozaev's identities
\[
\|\Delta Q\|^2_{L^2} =\frac{N\alpha}{4(\alpha+2)} \|Q\|^{\alpha+2}_{L^{\alpha+2}} = \frac{N\alpha}{8-(N-4)\alpha} \|Q\|^2_{L^2}.
\]
A direct computation shows that
\begin{align}
C_{\opt} &= \frac{4(\alpha+2)}{N\alpha} \left( \|\Delta Q\|_{L^2} \|Q\|^{\sigc}_{L^2} \right)^{-\frac{N\alpha-8}{4}}, \label{opt-con-GN} \\
E_0(Q)[M(Q)]^{\sigc} &= \frac{N\alpha-8}{2N\alpha} \left( \|\Delta Q\|_{L^2} \|Q\|^{\sigc}_{L^2}\right)^2, \label{rela-Q}
\end{align}
where $\sigc$ is as in \eqref{defi-sigc}.

\begin{lemma} \label{lem-coer-1}
	Let $N\geq 1$, $\mu\geq 0$ and $\frac{8}{N} <\alpha<\alpha^*$. Let $u_0 \in H^2$ satisfy \eqref{cond-ener}.
	\begin{itemize}
	\item If $u_0$ satisfies \eqref{cond-grad-gwp}, then the corresponding solution to the focusing problem \eqref{4NLS} satisfies
	\begin{align} \label{est-solu-gwp}
	\|\Delta u(t)\|_{L^2} \|u(t)\|_{L^2}^{\sigc} < \|\Delta Q\|_{L^2} \|Q\|_{L^2}^{\sigc}
	\end{align}
	for all $t$ in the existence time. In particular, the corresponding solution to  the focusing problem \eqref{4NLS} exists globally in time. Moreover, there exists $\rho=\rho(u_0,Q)>0$ such that
	\begin{align} \label{coer-1}
	\|\Delta u(t)\|_{L^2} \|u(t)\|_{L^2}^{\sigc} < (1-2\rho) \|\Delta Q\|_{L^2} \|Q\|_{L^2}^{\sigc}
	\end{align}
	for all $t\in \R$.
	\item If $u_0$ satisfies \eqref{cond-grad-blup}, then the corresponding solution to  the focusing problem \eqref{4NLS} satisfies
	\begin{align} \label{est-solu-blup}
	\|\Delta u(t)\|_{L^2} \|u(t)\|^{\sigc}_{L^2} > \|\Delta Q\|_{L^2} \|Q\|^{\sigc}_{L^2}
	\end{align}
	for all $t$ in the existence time.
	\end{itemize} 
\end{lemma}

\begin{proof}
	We only prove the first item, the second one is similar. Multiplying both sides of $E_\mu(u(t))$ by $[M(u(t))]^{\sigc}$ and using the Gagliardo-Nirenberg inequality together with $\mu\geq 0$, we have
	\begin{align}
	E_\mu(u(t)) [M(u(t))]^{\sigc} &= \frac{1}{2} \left( \|\Delta u(t)\|_{L^2} \|u(t)\|^{\sigc}_{L^2} \right)^2 + \frac{\mu}{2} \left( \|\nabla u(t)\|_{L^2} \|u(t)\|^{\sigc}_{L^2} \right)^2 -\frac{1}{\alpha+2} \|u(t)\|^{\alpha+2}_{L^{\alpha+2}} \|u(t)\|^{2 {\sigc}}_{L^2} \nonumber \\
	& \geq \frac{1}{2} \left( \|\Delta u(t)\|_{L^2} \|u(t)\|^{\sigc}_{L^2} \right)^2 -\frac{C_{\opt}}{\alpha+2} \|\Delta u(t)\|^{\frac{N\alpha}{4}}_{L^2} \|u(t)\|^{\frac{8-(N-4)\alpha}{4} +2\sigc}_{L^2} \nonumber \\
	& = \frac{1}{2} \left( \|\Delta u(t)\|_{L^2} \|u(t)\|^{\sigc}_{L^2}\right)^2 -\frac{C_{\opt}}{\alpha+2} \left(\|\Delta u(t)\|_{L^2} \|u(t)\|^{\sigc}_{L^2} \right)^{\frac{N\alpha}{4}} \nonumber \\
	&= g\left(\|\Delta u(t)\|_{L^2} \|u(t)\|^{\sigc}_{L^2}\right), \label{est-g}
	\end{align}
	where
	\[
	g(\lambda):= \frac{1}{2} \lambda^2 - \frac{C_{\opt}}{\alpha+2} \lambda^{\frac{N\alpha}{4}}.
	\]
	By Pohozaev's identities and \eqref{rela-Q}, a direct computation shows
	\begin{align} \label{energy-Q}
	g\left(\|\Delta Q\|_{L^2} \|Q\|^{\sigc}_{L^2}\right) = \frac{N\alpha-8}{2N\alpha} \left(\|\Delta Q\|_{L^2} \|Q\|^{\sigc}_{L^2}\right)^2 = E_0(Q) [M(Q)]^{\sigc}.
	\end{align}
	By \eqref{cond-ener}, the conservation of mass and energy, \eqref{est-g} and \eqref{energy-Q}, we infer that
	\[
	g\left(\|\Delta u(t)\|_{L^2} \|u(t)\|^{\sigc}_{L^2}\right) \leq E_\mu(u_0) [M(u_0)]^{\sigc} < E_0(Q) [M(Q)]^{\sigc} = g\left(\|\Delta Q\|_{L^2} \|Q\|^{\sigc}_{L^2}\right)
	\]
	for all $t$ in the existence time. By \eqref{cond-grad-gwp}, the continuity argument shows \eqref{est-solu-gwp}. Thus, by the conservation of mass and the local well-posedness, the corresponding solution exists globally in time. To see \eqref{coer-1}, we take $\theta = \theta(u_0,Q)>0$ such that
	\[
	E_\mu(u_0) [M(u_0)]^{\sigc} < (1-\theta) E_0(Q) [M(Q)]^{\sigc}.
	\]
	Using the fact that
	\[
	E_0(Q) [M(Q)]^{\sigc} = \frac{N\alpha-8}{2N\alpha} \left(\|\Delta Q\|_{L^2} \|Q\|^{\sigc}_{L^2}\right)^2 = \frac{N\alpha-8}{8(\alpha+2)} C_{\opt} \left(\|\Delta Q\|_{L^2} \|Q\|^{\sigc}_{L^2} \right)^{\frac{N\alpha}{4}},
	\]
	we get from 
	\[
	g\left(\|\Delta u(t)\|_{L^2} \|u(t)\|^{\sigc}_{L^2} \right) < (1-\theta) E_0(Q)[M(Q)]^{\sigc}
	\]
	that
	\begin{align} \label{coer-1-proo}
	\frac{N\alpha}{N\alpha-8} \left(\frac{\|\Delta u(t)\|_{L^2} \|u(t)\|^{\sigc}_{L^2}}{\|\Delta Q\|_{L^2} \|Q\|^{\sigc}_{L^2}} \right)^2 - \frac{8}{N\alpha-8} \left( \frac{\|\Delta u(t)\|_{L^2} \|u(t)\|^{\sigc}_{L^2}}{\|\Delta Q\|_{L^2} \|Q\|^{\sigc}_{L^2}}\right)^{\frac{N\alpha}{4}} < 1-\theta.
	\end{align}
	Consider the function $h(\lambda):= \frac{N\alpha}{N\alpha-8} \lambda^2 -\frac{8}{N\alpha-8} \lambda^{\frac{N\alpha}{4}}$ with $0<\lambda<1$. We see that $h$ is strictly increasing on $(0,1)$ and $h(0)=0, h(1)=1$. It follows from \eqref{coer-1-proo} that there exists $\rho=\rho(\theta)>0$ such that $\lambda<1-2\rho$. The proof is complete.
\end{proof}

\begin{remark} \label{rem-data}
	It follows from \eqref{est-g} and \eqref{energy-Q} that there is no $u_0 \in H^2$ satisfying \eqref{cond-ener} and 
	\[
	\|\Delta u_0\|_{L^2} \|u_0\|^{\sigc}_{L^2} = \|\Delta Q\|_{L^2} \|Q\|^{\sigc}_{L^2}.
	\]
\end{remark}

\begin{lemma} \label{lem-coer-2}
	Let $N\geq 1$, $\mu\geq 0$ and $\frac{8}{N}<\alpha<\alpha^*$. Let $u_0 \in H^2$ satisfy \eqref{cond-ener} and \eqref{cond-grad-gwp}. Let $\rho$ be as in \eqref{coer-1}. Then there exists $R_0 =R_0(\rho, u_0)>0$ such that for any $R\geq R_0$,
	\begin{align} \label{est-solu-gwp-ref}
	\|\Delta(\chi_R u(t))\|_{L^2} \|\chi_R u(t)\|^{\sigc}_{L^2} < (1-\rho) \|\Delta Q\|_{L^2} \|Q\|^{\sigc}_{L^2}
	\end{align}
	for all $t \in \R$, where $\chi_R(x) = \chi(x/R)$ with $\chi \in C^\infty_0(\R^N)$, $0\leq \chi \leq 1$ and
	\begin{align} \label{defi-chi}
	\chi(x) = \left\{
	\begin{array}{ccc}
	1 &\text{if}& |x| \leq \frac{1}{2}, \\
	0 &\text{if}& |x| \geq 1.
	\end{array}
	\right.
	\end{align}
	In particular, there exists $\nu=\nu (\rho)>0$ such that for any $R\geq R_0$,
	\begin{align} \label{coer-2}
	\|\Delta (\chi_R u(t))\|^2_{L^2} - \frac{N\alpha}{4(\alpha+2)} \|\chi_R u(t)\|^{\alpha+2}_{L^{\alpha+2}} \geq \nu \|\chi_R u(t)\|^{\alpha+2}_{L^{\alpha+2}}
	\end{align}
	for all $t\in \R$.
\end{lemma} 

\begin{proof}
	By the definition of $\chi_R$, we have $\|\chi_R u(t)\|_{L^2} \leq \|u(t)\|_{L^2}$. On the other hand, we see that
	\begin{align*}
	\int |\Delta(\chi f)|^2 dx &= \int |\chi \Delta f + 2\nabla \chi \cdot \nabla f + \Delta \chi f|^2 dx \\
	&= \int \chi^2 |\Delta f|^2 + 4 |\nabla \chi \cdot \nabla f|^2 + (\Delta \chi)^2 |f|^2 dx \\
	&\mathrel{\phantom{=}} + 4 \rea \int \chi \Delta f \nabla \chi \cdot \nabla \overline{f} dx + 2 \rea \int \chi \Delta f \Delta \chi \overline{f} dx + 4 \rea \int \nabla \chi \cdot \nabla f \Delta \chi \overline{f} dx.
	\end{align*}
	By integration by parts, we have
	\begin{align*}
	\rea \int \chi \Delta f \nabla \chi \cdot \nabla \overline{f} dx &= \sum_{k,l} \rea \int \chi \partial^2_k f \partial_l \chi \partial_l\overline{f} dx \\
	&= - \sum_{k,l} \rea \int \partial_k f \partial_k (\chi \partial_l \chi \partial_l \overline{f}) dx \\
	&= - \int |\nabla \chi \cdot \nabla f|^2 dx - \sum_{k,l} \rea \int \partial_k f \chi \partial^2_{kl} \chi \partial_l \overline{f} dx - \sum_{k,l} \rea \int \partial_k f \chi \partial_l \chi \partial^2_{kl} \overline{f} dx.
	\end{align*}
	We also have
	\begin{align*}
	\sum_{k,l} \rea \int \partial_k f \chi \partial_l \chi \partial^2_{kl} \overline{f} dx &= - \sum_{k,l} \rea \int \partial_l(\partial_k f \chi \partial_l \chi) \partial_k \overline{f} dx \\
	&= -\sum_{k,l} \rea \int \partial^2_{kl} f \chi \partial_l \chi \partial_k \overline{f} dx - \int |\nabla \chi|^2 |\nabla f|^2 dx - \int \chi \Delta \chi |\nabla f|^2 dx
	\end{align*}
	or
	\[
	\sum_{k,l} \rea \int \partial_k f \chi \partial_l \partial^2_{kl} \overline{f} dx = -\frac{1}{2} \int |\nabla \chi|^2 |\nabla f|^2 dx - \frac{1}{2} \int \chi \Delta \chi |\nabla f|^2 dx.
	\]
	It follows that
	\begin{align*}
	\rea \int \chi \Delta f \nabla \chi \cdot \nabla \overline{f} dx = &- \int |\nabla \chi \cdot \nabla f|^2 dx - \sum_{k,l} \rea \int \chi \partial_k f \partial^2_{kl} \chi \partial_l \overline{f} dx \\
	&+\frac{1}{2} \int |\nabla \chi|^2 |\nabla f|^2 dx +\frac{1}{2} \int \chi \Delta \chi |\nabla f|^2 dx.
	\end{align*}
	Thus
	\begin{align*}
	\int |\Delta(\chi f)|^2 dx = &\int \chi^2 |\Delta f|^2 dx + \int (\Delta \chi)^2 |f|^2 dx - 4\sum_{k,l} \rea \int \chi \partial_k f \partial^2_{kl} \chi \partial_l \overline{f} dx + 2 \int |\nabla \chi|^2 |\nabla f|^2 dx \\
	&+ 2\int \chi \Delta \chi |\nabla f|^2 dx + 2 \rea \int \chi \Delta f \Delta \chi \overline{f} dx + 4 \rea \int \nabla \chi \cdot \nabla f \Delta \chi \overline{f} dx.
	\end{align*}
	By H\"older's inequality, we have
	\begin{align*}
	\int |\nabla \chi|^2 |\nabla f|^2 dx &\leq \|\nabla \chi \|^2_{L^\infty} \|\nabla f\|^2_{L^2} \leq \|\nabla \chi\|^2_{L^\infty} \|\Delta f\|_{L^2} \|f\|_{L^2}, \\
	\int (\Delta \chi)^2 |f|^2 dx &\leq \|\Delta \chi\|^2_{L^\infty} \|f\|^2_{L^2}, \\
	\int \chi \Delta \chi |\nabla f|^2 dx &\leq \|\chi\|_{L^\infty} \|\Delta \chi\|_{L^\infty} \|\nabla f\|^2_{L^2} \leq \|\chi\|_{L^\infty} \|\Delta \chi\|_{L^\infty} \|\Delta f\|_{L^2} \|f\|_{L^2},
	\end{align*}
	and
	\begin{align*}
	\left| \rea \int \chi \partial_k f \partial^2_{kl} \chi \partial_l \overline{f} dx\right| &\leq \|\chi\|_{L^\infty} \|\partial^2_{kl} \chi\|_{L^\infty} \|\partial_k f\|_{L^2} \|\partial_l f\|_{L^2} \leq \|\chi\|_{L^\infty} \|\partial^2_{kl} \chi\|_{L^\infty} \|\Delta f\|_{L^2} \|f\|_{L^2}, \\
	\left|\rea \int \chi \Delta f \Delta \chi \overline{f} dx \right| &\leq \|\chi\|_{L^\infty} \|\Delta \chi\|_{L^\infty} \|\Delta f\|_{L^2} \|f\|_{L^2}, \\
	\left|\rea \int \nabla \chi \cdot \nabla f \Delta \chi \overline{f} dx \right| &\leq \|\nabla \chi\|_{L^\infty} \|\Delta \chi\|_{L^\infty} \|\nabla f\|_{L^2} \|f\|_{L^2} \leq \|\nabla \chi\|_{L^\infty} \|\Delta \chi\|_{L^\infty} \|\Delta f\|^{\frac{1}{2}}_{L^2} \|f\|^{\frac{3}{2}}_{L^2}.
	\end{align*}
	We thus get
	\begin{align*} 
	\int |\Delta(\chi f)|^2 dx \leq \int \chi^2 |\Delta f|^2 & + \left(4 \sum_{k,l} \|\chi\|_{L^\infty} \|\partial^2_{kl} \chi\|_{L^\infty} + 2 \|\nabla \chi\|^2_{L^\infty} + \|\chi\|_{L^\infty} \|\Delta \chi\|_{L^\infty} \right) \|\Delta f\|_{L^2} \|f\|_{L^2} \\
	&+ \|\Delta \chi\|^2_{L^\infty} \|f\|^2_{L^2} + \|\nabla \chi\|_{L^\infty} \|\Delta \chi\|_{L^\infty} \|\Delta f\|^{\frac{1}{2}}_{L^2} \|f\|^{\frac{3}{2}}_{L^2}.
	\end{align*}
	This together with \eqref{est-solu-gwp} imply that
	\begin{align} \label{proper-chi}
	\|\Delta(\chi_R u(t))\|^2_{L^2} = \int \chi^2_R |\Delta u(t)|^2 dx + C(u_0,Q) R^{-2},
	\end{align}
	where we have used the fact that
	\[
	\|\chi_R\|_{L^\infty} \lesssim 1, \quad \|\nabla \chi_R\|_{L^\infty} \lesssim R^{-1}, \quad \|\Delta \chi_R\|_{L^\infty} \lesssim R^{-2}.
	\]
	It follows from \eqref{coer-1} that
	\begin{align*}
	\|\Delta(\chi_R u(t))\|_{L^2} \|\chi_R u(t)\|^{\sigc}_{L^2} &\leq \left(\|\Delta u(t)\|^2_{L^2} + C(u_0,Q) R^{-2}\right)^{\frac{1}{2}} \|u(t)\|^{\sigc}_{L^2} \\
	&\leq \|\Delta u(t)\|_{L^2} \|u(t)\|^{\sigc}_{L^2} + C(u_0,Q) R^{-1} \\
	&\leq (1-2\rho) \|\Delta Q\|_{L^2} \|Q\|^{\sigc}_{L^2} + C(u_0,Q) R^{-1} \\
	&\leq (1-\rho) \|\Delta Q\|_{L^2} \|Q\|^{\sigc}_{L^2}
	\end{align*}
	provided that $R>0$ is taken sufficiently large depending on $u_0$ and $Q$. This proves \eqref{est-solu-gwp-ref}.
	
	The estimate \eqref{coer-2} follows from \eqref{est-solu-gwp-ref} and the following observation: if 
	\begin{align} \label{coer-2-proof-1}
	\|\Delta f\|_{L^2} \|f\|^{\sigc}_{L^2} <(1-\rho) \|\Delta Q\|_{L^2} \|Q\|_{L^2}^{\sigc},
	\end{align} 
	then there exists $\nu=\nu(\rho)>0$ such that
	\[
	K_0(f):=\|\Delta f\|^2_{L^2} -\frac{N\alpha}{4(\alpha+2)} \|f\|^{\alpha+2}_{L^{\alpha+2}} \geq \nu \|f\|^{\alpha+2}_{L^{\alpha+2}}.
	\]
	To see this, we have from the Gagliardo-Nirenberg inequality, \eqref{coer-2-proof-1} and \eqref{opt-con-GN} that
	\begin{align*}
	E_0(f) &=\frac{1}{2} \|\Delta f\|^2_{L^2} - \frac{1}{\alpha+2} \|f\|^{\alpha+2}_{L^{\alpha+2}} \\
	&\geq \frac{1}{2} \|\Delta f\|^2_{L^2} - \frac{C_{\opt}}{\alpha+2} \|\Delta f\|^{\frac{N\alpha}{4}}_{L^2} \|f\|^{\frac{8-(N-4)\alpha}{4}}_{L^2} \\
	&= \frac{1}{2} \|\Delta f\|^2_{L^2} \left( 1- \frac{2 C_{\opt}}{\alpha+2} \|\Delta f\|^{\frac{N\alpha-8}{4}}_{L^2} \|f\|^{\frac{8-(N-4)\alpha}{4}}_{L^2} \right) \\
	&= \frac{1}{2} \|\Delta f\|^2_{L^2} \left( 1- \frac{2 C_{\opt}}{\alpha+2} \left(\|\Delta f\|_{L^2} \|f\|^{\sigc}_{L^2} \right)^{\frac{N\alpha-8}{4}} \right) \\
	&> \frac{1}{2} \|\Delta f\|^2_{L^2} \left( 1- \frac{2C_{\opt}}{\alpha+2} (1-\rho)^{\frac{N\alpha-8}{4}} \left( \|\Delta Q\|_{L^2} \|Q\|^{\sigc}_{L^2} \right)^{\frac{N\alpha-8}{4}} \right) \\
	&= \frac{1}{2} \|\Delta f\|^2_{L^2} \left( 1-\frac{8}{N\alpha} (1-\rho)^{\frac{N\alpha-8}{4}} \right).
	\end{align*}
	It follows that
	\[
	\|\Delta f\|^2_{L^2} >\frac{N\alpha}{4(\alpha+2)} \frac{1}{(1-\rho)^{\frac{N\alpha-8}{4}}} \|f\|^{\alpha+2}_{L^{\alpha+2}}.
	\]
	We thus get
	\begin{align*}
	K_0(f) &= \frac{N\alpha}{4} E_0(f) - \frac{N\alpha-8}{8} \|\Delta f\|^2_{L^2} \\
	&> \frac{N\alpha}{8} \|\Delta f\|^2_{L^2} \left( 1 - \frac{8}{N\alpha} (1-\rho)^{\frac{N\alpha-8}{4}} \right) - \frac{N\alpha -8}{8} \|\Delta f\|^2_{L^2} \\
	&= \left(1- (1-\rho)^{\frac{N\alpha-8}{4}} \right) \|\Delta f\|^2_{L^2} \\
	&> \frac{N\alpha\left[1-(1-\rho)^{\frac{N\alpha-8}{4}} \right] }{4(\alpha+2)(1-\rho)^{\frac{N\alpha-8}{4}}} \|f\|^{\alpha+2}_{L^{\alpha+2}}
	\end{align*}
	which proves the observation. The proof of Lemma \ref{lem-coer-2} is now complete.
\end{proof}

\begin{remark} \label{rem-chi-R}
	It follows directly from H\"older's inequality and the definition of $\chi_R$ that 
	\[
	\|\Delta(\chi_R u(t))\|^2_{L^2} = \int \chi^2_R |\Delta u(t)|^2 dx + C(u_0,Q) R^{-1}.
	\]
	However, we need the refined estimate \eqref{proper-chi} for the later purpose. The decay $R^{-1}$ is not enough to show the space-time estimate \eqref{mora-est}.
\end{remark}

\subsection{Morawetz estimate}
Let us start with the following virial identity.
\begin{lemma}[Virial identity \cite{BL}] \label{lem-virial-iden}
	Let $N\geq 1$, $\mu\geq 0$ and $0<\alpha<\alpha^*$. Let $\varphi: \R^N \rightarrow \R$ be a sufficiently smooth and decaying function. Let $u$ be a $H^2$ solution to  the focusing problem \eqref{4NLS}. Define
	\[
	M_\varphi(t):= 2 \int \nabla \varphi \cdot \ima \left( \overline{u}(t) \nabla u(t)\right) dx.
	\]
	Then we have that
	\begin{align*}
	\frac{d}{dt} M_\varphi(t) &= \int \Delta^3 \varphi |u(t)|^2 dx - 2\int \Delta^2 \varphi |\nabla u(t)|^2 dx + 8 \sum_{k,l,m} \int \partial^2_{lm} \varphi \partial^2_{kl} \overline{u}(t) \partial^2_{mk} u(t) dx \\
	&\mathrel{\phantom{= \int \Delta^3 \varphi |u(t)|^2 dx}} - 4\sum_{k,l} \int \partial^2_{kl} \Delta \varphi \partial_k \overline{u}(t) \partial_l u(t) dx + 4\mu \sum_{k,l} \int \partial^2_{kl}\varphi \partial_k \overline{u}(t) \partial_l u(t) dx \\
	&\mathrel{\phantom{= \int \Delta^3 \varphi |u(t)|^2 dx}}- \mu \int \Delta^2 \varphi |u(t)|^2 dx -\frac{2\alpha}{\alpha+2} \int \Delta \varphi |u(t)|^{\alpha+2} dx.
	\end{align*}
\end{lemma}

\begin{remark} 
	In the case $\varphi(x) =|x|^2$, we have
	\[
	\frac{d}{dt} M_{|x|^2} (t) = 16 K_\mu(u(t)),
	\]
	where
	\begin{align} \label{defi-K-mu}
	\begin{aligned}
	K_\mu(u):&=\|\Delta u\|^2_{L^2} + \frac{\mu}{2} \|\nabla u\|^2_{L^2} -\frac{N\alpha}{4(\alpha+2)} \|u\|^{\alpha+2}_{L^{\alpha+2}} \\
	&= \frac{N\alpha}{4} E_\mu(u) - \frac{N\alpha -8}{8} \|\Delta u\|^2_{L^2} - \frac{(N\alpha-4)\mu}{8} \|\nabla u\|^2_{L^2}.
	\end{aligned}
	\end{align}
\end{remark}

\noindent {\it Proof of Lemma $\ref{lem-virial-iden}$.}
	The proof is essentially given in \cite[Lemma 3.1]{BL}. For the reader's convenience, we provide some details. We write
	\[
	M_\varphi(t) = \scal{u(t), \Gamma_\varphi u(t)},
	\]
	where
	\[
	\Gamma_\varphi:= -i (\nabla \varphi \cdot \nabla + \nabla \cdot \nabla \varphi).
	\]
	Note that if $u$ solves $i\partial_t u = H u$, then
	\[
	\frac{d}{dt} \scal{u, Au} = i \scal{u, [H,A]u} = \scal{u, [H, iA] u},
	\]
	where $[H,A]= HA - AH$ is the commutator operator. Applying the above fact with 
	\[
	H= \Delta^2 - \mu \Delta - |u|^\alpha,
	\]
	we get
	\[
	\frac{d}{dt} M_\varphi(t) = \scal{u, [\Delta^2, i \Gamma_\varphi] u} + \scal{u, [-\mu \Delta, i\Gamma_\varphi] u} + \scal{u, [-|u|^\alpha, i\Gamma_\varphi] u} =: \text{I} + \text{II} +\text{III}.
	\]
	\underline{Compute I.} We have
	\begin{align*}
	[\Delta^2, i\Gamma_\varphi] &= \Delta [\Delta, i\Gamma_\varphi] + [\Delta, i \Gamma_\varphi] \Delta \\
	&= \sum_{k} 2\partial_k [\Delta, i\Gamma_\varphi] \partial_k + [\partial_k, [\partial_k, [\Delta, i\Gamma_\varphi]]],
	\end{align*}
	where we have used the fact that
	\[
	\Delta A + A \Delta = \sum_k 2 \partial_k A \partial_k + [\partial_k, [\partial_k,A]]
	\]
	for an operator $A$. We also have
	\begin{align} \label{iden-varphi}
	[\Delta, i\Gamma_\varphi] = [\Delta, \nabla \varphi \cdot \nabla + \nabla \cdot \nabla \varphi] = 4 \sum_{l,m} \partial_l (\partial^2_{lm} \varphi) \partial_m + \Delta^2 \varphi.
	\end{align}
	It follows that
	\[
	[\Delta^2, i\Gamma_\varphi] = 8 \sum_{k,l,m} \partial^2_{kl} (\partial^2_{lm} \varphi) \partial^2_{mk} + 4\sum_{k,l} \partial_k (\partial^2_kl \Delta \varphi) \partial_l + 2\sum_{k,l} \partial_k (\Delta^2 \varphi) \partial_l + \Delta^3 \varphi
	\]
	hence
	\begin{align*}
	\text{I} &= \scal{u, [\Delta^2, i\Gamma_\varphi] u} \\
	&= 8 \sum_{k,l,m} \int \partial^2_{lm} \varphi \partial^2_{kl} \overline{u} \partial^2_{mk} u dx - 4 \sum_{k,l} \int \partial^2_{kl}\Delta \varphi \partial_k \overline{u} \partial_l u dx - 2\int \Delta^2 \varphi |\nabla u|^2 dx + \int \Delta^3\varphi |u|^2 dx.
	\end{align*}
	\underline{Compute II.} Using \eqref{iden-varphi}, we have
	\[
	\text{II} = \scal{u, [-\mu\Delta, i\Gamma_\varphi]u} = 4\mu \sum_{k,l} \int \partial^2_{kl}\varphi \partial_k \overline{u} \partial_l u dx - \mu \int \Delta^2 \varphi |u|^2 dx.
	\]
	\underline{Compute III.} We have
	\begin{align*}
	[-|u|^\alpha, i\Gamma_\varphi] u &= - [|u|^\alpha, \nabla \varphi \cdot \nabla + \nabla \cdot \nabla \varphi] u \\
	&= - \left(|u|^\alpha (\nabla \varphi \cdot \nabla u + \nabla \cdot (\nabla \varphi u)) - \nabla \varphi \cdot \nabla (|u|^\alpha u) - \nabla \cdot \left( \nabla \varphi |u|^\alpha u\right) \right) \\
	&= 2 \nabla \varphi \cdot \nabla (|u|^\alpha )u.
	\end{align*}
	It follows that
	\begin{align*}
	\text{III} &= \scal{u,[-|u|^\alpha, i\Gamma_\varphi]u} = 2 \int \nabla \varphi \cdot \nabla (|u|^\alpha) |u|^2 dx = -\frac{2\alpha}{\alpha+2} \int \Delta \varphi |u|^{\alpha+2} dx.
	\end{align*}
	Collecting the identities for I, II, and III, we complete the proof.
\hfill $\Box$

\vspace{3mm}

Let $\zeta: [0,\infty) \rightarrow [0,2]$ be a smooth function satisfying
\begin{align*} 
\zeta(r) = \left\{
\begin{array}{ccl}
2 & \text{if} & 0\leq r \leq 1, \\
0 &\text{if} & r\geq 2.
\end{array}
\right.
\end{align*}
We define the function $\theta: [0,\infty) \rightarrow [0,\infty)$ by
\[
\theta(r):= \int_0^r \int_0^s \zeta(z)dz ds.
\]
Given $R>0$, we define a radial function
\begin{align} \label{def-varphi-R}
\varphi_R(x) = \varphi_R(r):= R^2 \theta(r/R), \quad r=|x|.
\end{align}
We readily check that
\begin{align} \label{proper-varphi-R}
2 \geq \varphi''_R(r) \geq 0, \quad 2-\frac{\varphi'(r)}{r} \geq 0, \quad 2N-\Delta \varphi_R(x) \geq 0, \quad \forall r \geq 0, \quad \forall x \in \R^N.
\end{align}
%We also have that
%\[
%\|\nabla^k \varphi_R\|_{L^\infty} \lesssim R^{2-k}, \quad k=0,\cdots, 6
%\]
%and
%\[
%\supp(\nabla^k \varphi_R) \subset \left\{
%\begin{array}{cl}
%\{|x| \leq 2R\} &\text{if } k=1,2, \\
%\{R \leq |x| \leq 2R\} &\text{if } k=3,\cdots, 6.
%\end{array}
%\right.
%\]

\begin{proposition} \label{prop-mora-est}
	Let $N\geq 2$, $\mu\geq 0$ and $\frac{8}{N}<\alpha<\alpha^*$. Let $u_0\in H^2$ be radially symmetric satisfying \eqref{cond-ener} and \eqref{cond-grad-gwp}. Then, for any time interval $I \subset \R$, the corresponding solution to  the focusing problem \eqref{4NLS} satisfies 
	\begin{align} \label{mora-est}
	\int_I \|u(t)\|^{\alpha+2}_{L^{\alpha+2}} dt \leq C(u_0,Q) |I|^{\frac{1}{3}}
	\end{align}
	for some constant $C(u_0,Q)$ depending only on $u_0$ and $Q$.
\end{proposition}

\begin{proof}
	Let $\varphi_R$ be as in \eqref{def-varphi-R}. By the Cauchy-Schwarz inequality, the conservation of mass and \eqref{est-solu-gwp}, we see that
	\begin{align} \label{mora-est-proof-1}
	|M_{\varphi_R}(t)| \leq \|\nabla \varphi_R\|_{L^\infty} \|u(t)\|_{L^2} \|\nabla u(t)\|_{L^2} \leq \|\nabla \varphi_R\|_{L^\infty} \|u(t)\|^{\frac{3}{2}}_{L^2} \|\Delta u(t)\|^{\frac{1}{2}}_{L^2} \lesssim R
	\end{align}
	for all $t\in \R$, where the implicit constant depends only on $u_0$ and $Q$. By Lemma $\ref{lem-virial-iden}$ and the fact that $\varphi_R(x)=|x|^2$ for $|x| \leq R$, we have
	\begin{align*}
	\frac{d}{dt} M_{\varphi_R}(t) &= \int \Delta^3 \varphi_R |u(t)|^2 dx - 2\int \Delta^2 \varphi_R |\nabla u(t)|^2 dx + 8 \sum_{k,l,m} \int \partial^2_{lm} \varphi_R \partial^2_{kl} \overline{u}(t) \partial^2_{mk} u(t) dx \\
	&\mathrel{\phantom{= \int \Delta^3 \varphi_R |u(t)|^2 dx}} - 4\sum_{k,l} \int \partial^2_{kl} \Delta \varphi_R \partial_k \overline{u}(t) \partial_l u(t) dx + 4\mu \sum_{k,l} \int \partial^2_{kl}\varphi_R \partial_k \overline{u}(t) \partial_l u(t) dx \\
	&\mathrel{\phantom{= \int \Delta^3 \varphi_R |u(t)|^2 dx}}- \mu \int \Delta^2 \varphi_R |u(t)|^2 dx -\frac{2\alpha}{\alpha+2} \int \Delta \varphi_R |u(t)|^{\alpha+2} dx \\
	&= 16 \int_{|x| \leq R} |\Delta u(t)|^2 dx -\frac{4N\alpha}{\alpha+2} \int_{|x| \leq R} |u(t)|^{\alpha+2} dx \\
	&\mathrel{\phantom{=}} + \int \Delta^3 \varphi_R |u(t)|^2 dx - 2 \int \Delta^2\varphi_R |\nabla u(t)|^2 dx + 8 \sum_{k,l,m} \int_{|x|>R} \partial^2_{lm} \varphi_R \partial^2_{kl} \overline{u}(t) \partial^2_{mk} u(t) dx \\
	&\mathrel{\phantom{=}} - 4\sum_{k,l} \int \partial^2_{kl} \Delta \varphi_R \partial_k \overline{u}(t) \partial_l u(t) dx + 4\mu \sum_{k,l} \int \partial^2_{kl} \varphi_R \partial_k \overline{u}(t) \partial_l u(t) dx \\
	&\mathrel{\phantom{=}}- \mu \int \Delta^2\varphi_R |u(t)|^2 dx - \frac{2\alpha}{\alpha+2} \int_{|x|>R} \Delta \varphi_R |u(t)|^{\alpha+2} dx.
	\end{align*}
	Using \eqref{est-solu-gwp} and H\"older's inequality, we have
	\begin{align*}
	\left| \int \Delta^3 \varphi_R |u(t)|^2 dx \right| &\lesssim \|\Delta^3 \varphi_R\|_{L^\infty} \|u(t)\|^2_{L^2} \lesssim R^{-4}, \\
	\left| \int \Delta^2 \varphi_R |\nabla u(t)|^2 dx \right| &\lesssim \|\Delta^2\varphi_R\|_{L^\infty} \|\nabla u(t)\|^2_{L^2} \lesssim \|\Delta^2\varphi_R\|_{L^\infty} \|\Delta u(t)\|_{L^2} \|u(t)\|_{L^2} \lesssim R^{-2},
	\end{align*}
	and
	\begin{align*}
	\left| \int \Delta^2 \varphi_R |u(t)|^2 dx \right| &\lesssim \|\Delta^2 \varphi_R\|_{L^\infty} \|u(t)\|^2_{L^2} \lesssim R^{-2}, \\
	\left| \int \partial^2_{kl} \Delta \varphi_R \partial_k \overline{u}(t) \partial_l u(t) dx \right| &\lesssim \|\partial^2_{kl}\Delta\varphi_R\|_{L^\infty} \|\partial_k u(t)\|_{L^2} \|\partial_l u(t)\|_{L^2} \lesssim R^{-2}.
	\end{align*}
	Since $u$ is radial, we use the fact that
	\[
	\partial^2_{jk} = \left(\frac{\delta_{jk}}{r} - \frac{x_j x_k}{r^3} \right) \partial_r + \frac{x_j x_k}{r^2} \partial^2_r
	\]
	and \eqref{proper-varphi-R} to get
	\begin{align*}
	\sum_{k,l,m} \partial^2_{lm} \varphi_R  \partial^2_{kl} \overline{u} \partial^2_{mk} u &= \varphi''_R |\partial^2_r u|^2  + \frac{N-1}{r^3} \varphi'_R |\partial_r u|^2 \\
	& \geq \frac{N-1}{r^3} \varphi'_R |\partial_r u|^2. 
	\end{align*}
	It follows that
	\[
	\sum_{k,l,m} \int_{|x|>R} \partial^2_{lm} \varphi_R \partial^2_{kl} \overline{u} \partial^2_{mk} u dx  \geq \int_{R\leq |x| \leq 2R} \frac{N-1}{r^2} \frac{\varphi'_R}{r} |\partial_r u|^2 dx.
	\]
	We also have
	\[
	\sum_{k,l} \partial^2_{kl} \varphi_R \partial_k \overline{u} \partial_l u = \varphi''_R |\partial_r u|^2 \geq 0.
	\]
	We thus get
	\begin{align*}
	\frac{d}{dt} M_{\varphi_R}(t) &\geq 16 \left( \int_{|x| \leq R} |\Delta u(t)|^2 dx - \frac{N\alpha}{4(\alpha+2)} \int_{|x| \leq R} |u(t)|^{\alpha+2} dx \right) \\
	&\mathrel{\phantom{\geq}} + 8 \int_{|x|>R} \frac{N-1}{r^2} \frac{\varphi'_R}{r} |\nabla u(t)|^2 dx - \frac{2\alpha}{\alpha+2} \int_{|x|>R} \Delta \varphi_R |u(t)|^{\alpha+2} dx + O(R^{-2}+R^{-4}).
	\end{align*}
	By \eqref{est-solu-gwp} and \eqref{proper-varphi-R}, we have
	\[
	\int_{|x|>R} \frac{N-1}{r^2} \frac{\varphi'_R}{r} |\nabla u(t)|^2 dx \lesssim R^{-2} \|\nabla u(t)\|^2_{L^2} \lesssim R^{-2} \|\Delta u(t)\|_{L^2} \|u(t)\|_{L^2} \lesssim R^{-2}.
	\]
	Using the fact that $\|\Delta \varphi_R\|_{L^\infty} \lesssim 1$ and the radial Sobolev embedding (see \cite{Strauss}): for $N\geq 2$,
	\begin{align} \label{rad-sobo-emb}
	\sup_{x\ne 0} |x|^{\frac{N-1}{2}} |f(x)| \leq C(N) \|\nabla f\|^{\frac{1}{2}}_{L^2} \|f\|^{\frac{1}{2}}_{L^2}, \quad \forall f\in H^1_{\rad},
	\end{align}
	we have
	\begin{align*}
	\int_{|x|>R} \Delta \varphi_R |u(t)|^{\alpha+2} dx &\lesssim \|u(t)\|^\alpha_{L^\infty(|x|>R)} \|u(t)\|^2_{L^2} \\
	&\lesssim R^{-\frac{(N-1)\alpha}{2}} \|\nabla u(t)\|^{\frac{\alpha}{2}}_{L^2} \|u(t)\|^{2+\frac{\alpha}{2}}_{L^2} \\
	&\lesssim R^{-\frac{(N-1)\alpha}{2}} \|\Delta u(t)\|^{\frac{\alpha}{4}}_{L^2} \|u(t)\|^{2+\frac{3\alpha}{4}}_{L^2} \\
	&\lesssim R^{-\frac{(N-1)\alpha}{2}} \lesssim R^{-2}, 
	\end{align*}
	where, in the last estimate, we have used the fact that $(N-1)\alpha >4$ as $\alpha>\frac{8}{N}$. 
	We thus obtain
	\begin{align} \label{mora-est-proof-2}
	\frac{d}{dt}M_{\varphi_R}(t) \geq 16 \left(\int_{|x| \leq R} |\Delta u(t)|^2 dx - \frac{N\alpha}{4(\alpha+2)} \int_{|x| \leq R} |u(t)|^{\alpha+2} dx \right) + O\left( R^{-2} \right)
	\end{align}
	for all $t \in \R$. On the other hand, we have from \eqref{proper-chi} that
	\begin{align*}
	\int |\Delta(\chi_R u)|^2 dx &= \int \chi^2_R |\Delta u|^2 dx +  O\left(R^{-2}\right) \\
	&= \int_{|x| \leq R} |\Delta u|^2 dx - \int_{R/2 \leq |x| \leq R} (1-\chi^2_R) |\Delta u|^2 dx + O\left(R^{-2}\right).
	\end{align*}
	We also have
	\begin{align*}
	\int|\chi_R u|^{\alpha+2} dx = \int_{|x| \leq R} |u|^{\alpha+2} - \int_{R/2 \leq |x| \leq R} (1-\chi^{\alpha+2}_R) |u|^{\alpha+2}dx.
	\end{align*}
	This implies that
	\begin{align*}
	\int_{|x| \leq R} |\Delta u|^2 dx &- \frac{N\alpha}{4(\alpha+2)} \int_{|x| \leq R} |u|^{\alpha+2} dx \\
	&= \int |\Delta(\chi_R u)|^2 dx -\frac{N\alpha}{4(\alpha+2)} \int |\chi_R u|^{\alpha+2} dx + \int (1-\chi^2_R) |\Delta u|^2 dx \\
	&\mathrel{\phantom{=}} -\int_{R/2 \leq |x| \leq R} (1-\chi^{\alpha+2}_R) |u|^{\alpha+2} + O\left(R^{-2}\right).
	\end{align*}
	The radial Sobolev embedding \eqref{rad-sobo-emb} together with $0\leq \chi_R \leq 1$ imply 
	\[
	\int_{|x| \leq R} |\Delta u|^2 dx - \frac{N\alpha}{4(\alpha+2)} \int_{|x| \leq R} |u|^{\alpha+2} dx \geq \|\Delta(\chi_R u)\|^2_{L^2} - \frac{N\alpha}{4(\alpha+2)} \|\chi_R u\|^{\alpha+2}_{L^{\alpha+2}} + O \left( R^{-2}\right).
	\]
	We thus get from \eqref{mora-est-proof-2} that
	\[
	\frac{d}{dt} M_{\varphi_R}(t) \geq 16 \left(\|\Delta(\chi_R u(t))\|^2_{L^2} - \frac{N\alpha}{4(\alpha+2)} \|\chi_R u(t)\|^{\alpha+2}_{L^{\alpha+2}} \right) + O  \left( R^{-2} \right).
	\]
	By Lemma $\ref{lem-coer-2}$, there exist $R_0=R_0(u_0,Q)>0$ and $\nu = \nu(u_0,Q)>0$ such that for any $R\geq R_0$,
	\[
	16 \nu \|\chi_R u(t)\|^{\alpha+2}_{L^{\alpha+2}} \leq \frac{d}{dt} M_{\varphi_R}(t) + O \left( R^{-2}\right)
	\]
	for all $t\in \R$. Taking the integration in time, we have for any $I \subset \R$,
	\[
	\int_I \|\chi_R u(t)\|^{\alpha+2}_{L^{\alpha+2}} dt \lesssim \sup_{t\in I} |M_{\varphi_R}(t)| + O \left( R^{-2}\right)|I|.
	\]
	By the definition of $\chi_R$ and \eqref{mora-est-proof-1}, we get
	\[
	\int_I \int_{|x| \leq R/2} |u(t,x)|^{\alpha+2} dx dt \lesssim R +  R^{-2} |I|.
	\]
	On the other hand, by the radial Sobolev embedding,
	\[
	\int_{|x| \geq R/2} |u(t,x)|^{\alpha+2} dx \leq \left(\sup_{|x|\geq R/2} |u(t,x)|^{\alpha}\right) \|u(t)\|^2_{L^2} \lesssim R^{-\frac{(N-1)\alpha}{2}} \lesssim R^{-2}.
	\]
	This shows that
	\[
	\int_I \|u(t)\|^{\alpha+2}_{L^{\alpha+2}} dt \lesssim R +   R^{-2} |I|.
	\]
	Taking $R= |I|^{\frac{1}{3}}$, we get for $|I|$ sufficiently large,
	\[
	\int_I \|u(t)\|^{\alpha+2}_{L^{\alpha+2}} dt \lesssim |I|^{\frac{1}{3}}.
	\]
	For $|I|$ sufficiently small, we simply use the Sobolev embedding and \eqref{est-solu-gwp} to have
	\[
	\int_I \|u(t)\|^{\alpha+2}_{L^{\alpha+2}} dt \lesssim \int_I \|u(t)\|^{\alpha+2}_{H^2} dt \lesssim |I| \lesssim |I|^{\frac{1}{3}}
	\]
	which completes the proof.
\end{proof}

\begin{corollary} \label{coro-mora-est}
	Let $N\geq 2$, $\mu\geq 0$ and $\frac{8}{N}<\alpha<\alpha^*$. Let $u_0 \in H^2$ be radially symmetric and satisfy \eqref{cond-ener} and \eqref{cond-grad-gwp}. Then there exists $t_n \rightarrow \infty$ such that the corresponding global solution to  the focusing problem \eqref{4NLS} satisfies for any $R>0$,
	\begin{align} \label{small-L2}
	\lim_{n\rightarrow \infty} \int_{|x| \leq R} |u(t_n,x)|^2 dx =0.
	\end{align}
\end{corollary}

\begin{proof}
	We first claim that 
	\[
	\liminf_{t\rightarrow \infty} \|u(t)\|_{L^{\alpha+2}} =0.
	\]
	In fact, assume that it is not true. Then there exist $t_0>0$ and $\varrho >0$ such that
	\[
	\|u(t)\|_{L^{\alpha+2}} \geq \varrho
	\]
	for all $t\geq t_0$. Taking $I\subset [t_0,\infty)$, we have
	\[
	\int_I \|u(t)\|^{\alpha+2}_{L^{\alpha+2}} dt \geq \varrho^{\alpha+2} |I|.
	\]
	This however contradicts \eqref{mora-est} for $|I|$ sufficiently large, and the claim is proved. 
	
	Thus, there exists $t_n \rightarrow \infty$ such that $\lim_{n\rightarrow \infty} \|u(t_n)\|_{L^{\alpha+2}} =0$. Now let $R>0$. By H\"older's inequality,
	\begin{align*}
	\int_{|x|\leq R} |u(t_n,x)|^2 dx &\leq \left( \int_{|x| \leq R} dx \right)^{\frac{\alpha}{\alpha+2}} \left(\int_{|x| \leq R} |u(t_n,x)|^{\alpha+2} dx\right)^{\frac{2}{\alpha+2}} \\
	&\lesssim R^{\frac{N\alpha}{\alpha+2}} \left(\int |u(t_n,x)|^{\alpha+2} dx\right)^{\frac{2}{\alpha+2}} \rightarrow 0
	\end{align*}
	as $n\rightarrow \infty$. The proof is complete.
\end{proof}
	
	\subsection{Energy scattering}
	In this section, we give the proof of Theorem $\ref{theo-scat}$ which follows from the following result.
	
	\begin{proposition} \label{prop-scat}
		Let $N\geq 2$, $\mu\geq 0$ and $\frac{8}{N}<\alpha<\alpha^*$. Let $u_0 \in H^2$ be radially symmetric and satisfy \eqref{cond-ener} and \eqref{cond-grad-gwp}. Then for any $\vareps>0$, there exists $T=T(\vareps, u_0,Q)$ sufficiently large such that the corresponding global solution to  the focusing problem \eqref{4NLS} satisfies
		\begin{align} \label{small-epsilon}
		\|e^{-i(t-T)(\Delta^2-\mu\Delta)} u(T)\|_{L^{\kbo}([T,\infty),L^{\rbo})} \lesssim \vareps^\upsilon
		\end{align}
		for some $\upsilon>0$, where $\kbo$ and $\rbo$ are as in \eqref{defi-qrkm}.
	\end{proposition}

\begin{proof}
	Let $T>0$ be a large parameter depending on $\vareps, u_0$ and $Q$ to be chosen later. For $T>\vareps^{-\sigma}$ with some $\sigma>0$ to be determined later, we use the Duhamel formula to write
	\begin{align} \label{duhamel}
	\begin{aligned}
	e^{-i(t-T)(\Delta^2-\mu\Delta)} u(T) &= e^{-it(\Delta^2 -\mu \Delta)} u_0 + i \int_0^T e^{-i(t-s)(\Delta^2-\mu \Delta)} |u(s)|^\alpha u(s) ds \\
	&= e^{-it(\Delta^2- \mu \Delta)} u_0 + F_1(t) + F_2(t),
	\end{aligned}
	\end{align}
	where
	\[
	F_1(t):= i \int_I e^{-i(t-s)(\Delta^2-\mu \Delta)} |u(s)|^\alpha u(s) ds, \quad F_2(t):= i \int_J e^{-i(t-s)(\Delta^2-\mu\Delta)} |u(s)|^\alpha u(s) ds
	\]
	with $I:= [T-\vareps^{-\sigma}, T]$ and $J:= [0,T-\vareps^{-\sigma}]$. 
	
	\noindent {\bf Estimate the linear part.} By Sobolev embedding and Strichartz estimates, 
	\[
	\|e^{-it(\Delta^2-\mu \Delta)} u_0\|_{L^{\kbo}(\R, L^{\rbo})} \lesssim \||\nabla|^{\gamc} e^{-it(\Delta^2-\mu\Delta)} u_0\|_{L^{\kbo}(\R, L^{\lbo})} \lesssim \|u_0\|_{\dot{H}^{\gamc}} \lesssim \|u_0\|_{H^2} <\infty,
	\]
	where
	\begin{align} \label{defi-l}
	\lbo := \frac{2N\alpha(\alpha+2)}{N\alpha^2+4(N-2)\alpha -16}.
	\end{align}
	Note that $(\kbo, \lbo)$ is a Biharmonic admissible pair. By the monotone convergence theorem, we can find $T>\vareps^{-\sigma}$ so that
	\begin{align} \label{est-linear}
	\|e^{-it(\Delta^2-\mu \Delta)} u_0\|_{L^{\kbo}([T,\infty),L^{\rbo})} \lesssim \vareps.
	\end{align}
	
	\noindent {\bf Estimate $F_1$.} 
	By Lemma $\ref{lem-non-est}$ and H\"older's inequality, we have
	\[
	\|F_1\|_{L^{\kbo}([T,\infty),L^{\rbo})} \lesssim \||u|^\alpha u\|_{L^{\mbo'}(I, L^{\rbo'})} \lesssim \|u\|^{\alpha+1}_{L^{\kbo}(I,L^{\rbo})} \lesssim |I|^{\frac{\alpha+1}{\kbo}} \|u\|^{\alpha+1}_{L^\infty(I,L^{\rbo})}.
	\]
	To estimate $\|u\|_{L^\infty(I,L^{\rbo})}$, we have from \eqref{small-L2} (by enlarging $T$ if necessary) that for any $R>0$,
	\[
	\int_{|x|\leq R} |u(T,x)|^2 dx \lesssim \vareps^2.
	\]
	By the definition of $\chi_R$, we get
	\[
	\int \chi_R(x) |u(T,x)|^2 dx \lesssim \vareps^2.
	\]
	We next compute
	\begin{align*}
	\frac{d}{dt} \int \chi_R |u(t)|^2 dx &= 2 \rea \int \chi_R \partial_t u(t) \overline{u}(t) dx \\
	&= 2 \ima \int \chi_R \left(\Delta^2 u(t) - \mu \Delta u(t)\right) \overline{u}(t) dx \\
	&= 2 \ima \int \Delta \chi_R \Delta u(t) \overline{u}(t) + 2 \nabla \chi_R \cdot \nabla \overline{u}(t) \Delta u(t) + \mu  \nabla \chi_R \cdot \nabla u(t) \overline{u}(t) dx.
	\end{align*}
	It follows from H\"older's inequality and \eqref{est-solu-gwp} that
	\begin{align*}
	\left| \frac{d}{dt} \int \chi_R |u(t)|^2 dx \right| &\lesssim 2\|\Delta \chi_R\|_{L^\infty} \|\Delta u(t)\|_{L^2} \|u(t)\|_{L^2} + 4 \|\nabla \chi_R\|_{L^\infty} \|\nabla u(t)\|_{L^2} \|\Delta u(t)\|_{L^2} \\
	&\mathrel{\phantom{\lesssim}} + 2\mu \|\nabla \chi_R\|_{L^\infty} \|\nabla u(t)\|_{L^2} \|u(t)\|_{L^2} \\
	&\lesssim R^{-1}
	\end{align*}
	for all $t\in \R$. We thus have for any $t\leq T$, 
	\begin{align*}
	\int \chi_R(x) |u(t,x)|^2 dx &= \int \chi_R(x) |u(T,x)|^2 dx - \int_t^T \left(\frac{d}{ds} \int \chi_R(x) |u(s,x)|^2 dx \right) ds \\
	&\leq \int \chi_R(x) |u(T,x)|^2 dx + CR^{-1}(T-t)
	\end{align*}
	for some constant $C=C(u_0,Q)>0$. By choosing $R>\vareps^{-2-\sigma}$, we see that for any $t\in I$,
	\[
	\int \chi_R(x) |u(t,x)|^2 dx \leq C\vareps^2 + CR^{-1} \vareps^{-\sigma} \leq 2C \vareps^2
	\]
	hence
	\begin{align} \label{est-chi-R}
	\|\chi_R u\|_{L^\infty(I,L^2)} \lesssim \vareps,
	\end{align}
	where we have used the fact that $\chi^2_R \leq \chi_R$ due to $0\leq \chi_R \leq 1$.
	
	In the case $N\geq 5$, we use \eqref{est-chi-R}, the radial Sobolev embedding, and \eqref{est-solu-gwp} to have
	\begin{align*}
	\|u\|_{L^\infty(I, L^{\rbo})} &\leq \|\chi_R u\|_{L^\infty(I,L^{\rbo})} + \|(1-\chi_R) u\|_{L^\infty(I,L^{\rbo})} \\
	&\leq \|\chi_R u\|^{\frac{8-(N-4)\alpha}{4(\alpha+2)}}_{L^\infty(I,L^2)} \|\chi_R u\|^{\frac{N\alpha}{4(\alpha+2)}}_{L^\infty(I, L^{\frac{2N}{N-4}})}  + \|(1-\chi_R) u\|^{\frac{\alpha}{\alpha+2}}_{L^\infty(I,L^\infty)} \|(1-\chi_R) u\|^{\frac{2}{\alpha+2}}_{L^\infty(I, L^2)} \\
	&\lesssim \vareps^{\frac{8-(N-4)\alpha}{4(\alpha+2)}} + R^{-\frac{(N-1)\alpha}{2(\alpha+2)}}  \\
	&\lesssim \vareps^{\frac{8-(N-4)\alpha}{4(\alpha+2)}}
	\end{align*}
	provided that $R>\vareps^{-\frac{8-(N-4)\alpha}{2(N-1)\alpha}}$. Here we have used the Sobolev embedding $H^2(\R^N) \hookrightarrow L^{\frac{2N}{N-4}}(\R^N)$ for $N\geq 5$.
	
	In the case $N=4$, we interpolate between $L^2$ and $L^{2\rbo}$ and use the Sobolev embedding $H^2(\R^4) \hookrightarrow L^{2\rbo}(\R^4)$ to have
	\[
	\|\chi_R u\|_{L^\infty(I,L^{\rbo})} \leq \|\chi_R u\|^{\frac{1}{\rbo-1}}_{L^\infty(I,L^2)} \|\chi_R u\|^{\frac{\rbo-2}{\rbo-1}}_{L^\infty(I,L^{2\rho})} \lesssim \vareps^{\frac{1}{\rbo-1}}.
	\]
	Thus we get
	\[
	\|u\|_{L^\infty(I,L^{\rho})} \lesssim \vareps^{\frac{1}{\rbo-1}}
	\]
	provided that $R>\vareps^{-\frac{2(\alpha+2)}{3\alpha(\rbo-1)}}$.
	
	In the case $2\leq N\leq 3$, we use the embedding $H^2(\R^N) \hookrightarrow L^\infty(\R^N)$ to get
	\[
	\|\chi_R u\|_{L^\infty(I,L^{\rbo})} \leq \|\chi_R u\|^{\frac{2}{\alpha+2}}_{L^\infty(I,L^2)} \|\chi_R u\|^{\frac{\alpha}{\alpha+2}}_{L^\infty(I, L^\infty)} \lesssim \vareps^{\frac{2}{\alpha+2}}.
	\]
	It follows that
	\[
	\|u\|_{L^\infty(I, L^{\rbo})} \lesssim \vareps^{\frac{2}{\alpha+2}}
	\]
	provided that $R>\vareps^{-\frac{4}{(N-1)\alpha}}$. 
	
	In all cases, by taking $R>0$ sufficiently large depending on $\vareps$, we have proved that
	\[
	\|u\|_{L^\infty(I, L^{\rbo})} \lesssim \vareps^{\delta},
	\]
	where
	\[
	\delta:= \left\{
	\renewcommand*{\arraystretch}{1.3}
	\begin{array}{ccl}
	\frac{8-(N-4)\alpha}{4(\alpha+2)} &\text{if}& N\geq 5, \\
	\frac{2}{\alpha+2} &\text{if} & 2\leq N\leq 4.
	\end{array}
	\right.
	\]
	This shows that
	\begin{align} \label{est-F1}
	\|F_1\|_{L^{\kbo}([T,\infty),L^{\rbo})} \lesssim  \vareps^{(\alpha+1)\left(\delta - \frac{\sigma}{\kbo}\right)}.
	\end{align}
	
	\noindent {\bf Estimate $F_2$.} We will consider separately three cases: $N\geq 6$, $N=5$, and $2\leq N \leq 4$.
	
	\vspace{2mm}
	
	\noindent {\bf Case 1. $N\geq 6$.} By H\"older's inequality,
	\[
	\|F_2\|_{L^{\kbo}([T,\infty), L^{\rbo})} \leq \|F_2\|^\theta_{L^{\kbo}([T,\infty), L^{\lbo})} \|F_2\|^{1-\theta}_{L^{\kbo}([T,\infty), L^{\nbo})}
	\]
	where $\lbo$ is as in \eqref{defi-l}, $\theta \in (0,1)$ and $\nbo>\rbo$ satisfy
	\[
	\frac{1}{\rbo} =\frac{\theta}{\lbo} + \frac{1-\theta}{\nbo}.
	\]
	Since $(\kbo,\lbo) \in B$, we use Strichartz estimates and the fact that
	\[
	F_2(t) = e^{-i(t-T+\vareps^{-\sigma})(\Delta^2-\mu \Delta)} u(T-\vareps^{-\sigma}) - e^{-it(\Delta^2-\mu \Delta)} u_0
	\]
	to have
	\[
	\|F_2\|_{L^{\kbo}([T, \infty), L^{\lbo})} \lesssim 1.
	\]
	By dispersive estimates \eqref{disper-est}, Sobolev embedding and \eqref{est-solu-gwp}, we have for $t\geq T$,
	\begin{align*}
	\|F_2(t)\|_{L^{\nbo}} &\lesssim \int_0^{T-\vareps^{-\sigma}} (t-s)^{-\frac{N}{4}\left(1-\frac{2}{\nbo}\right)} \||u(s)|^\alpha u(s)\|_{L^{{\nbo}'}} ds \\
	&\lesssim \int_0^{T-\vareps^{-\sigma}} (t-s)^{-\frac{N}{4} \left(1-\frac{2}{\nbo}\right)} \|u(s)\|^{\alpha+1}_{L^{(\alpha+1){\nbo}'}} ds \\
	&\lesssim (t-T+\vareps^{-\sigma})^{-\frac{N}{4} \left(1-\frac{2}{\nbo}\right) +1}
	\end{align*}
	provided that
	\[
	(\alpha+1){\nbo}' \in \left[2,\frac{2N}{N-4}\right], \quad \frac{N}{4} \left(1-\frac{2}{\nbo}\right) -1 >0.
	\]
	It follows that
	\begin{align*}
	\|F_2\|_{L^{\kbo}([T,\infty), L^{\nbo})} &\lesssim \left( \int_T^\infty (t-T+\vareps^{-\sigma})^{-\left[\frac{N}{4} \left(1-\frac{2}{\nbo}\right) -1 \right] \kbo} dt \right)^{\frac{1}{\kbo}} \\
	&\lesssim \vareps^{\sigma \left[ \frac{N}{4} \left(1-\frac{2}{\nbo}\right) -1 -\frac{1}{\kbo} \right]}
	\end{align*}
	for
	\[
	\frac{N}{4} \left(1-\frac{2}{\nbo}\right) -1 -\frac{1}{\kbo} >0.
	\]
	We thus obtain
	\begin{align} \label{est-F2}
	\|F_2\|_{L^{\kbo}([T,\infty), L^{\rbo})} \lesssim \vareps^{\sigma \left[\frac{N}{4}\left(1-\frac{2}{\nbo}\right) - 1 - \frac{1}{\kbo}\right] (1-\theta)}.
	\end{align}
	We will choose a suitable $\nbo$ satisfying
	\[
	\nbo > \rbo, \quad (\alpha+1) {\nbo}' \in \left[2,\frac{2N}{N-4}\right], \quad \frac{N}{4} \left(1-\frac{2}{\nbo}\right) - 1- \frac{1}{\kbo} >0.
	\]
	These condition are equivalent to 
	\begin{align} \label{cond-n}
	0 \leq \frac{1}{\nbo} \leq \frac{1}{\alpha+2}, \quad \frac{1}{\nbo} \in \left[ \frac{1-\alpha}{2}, \frac{N+4-(N-4)\alpha}{2N}\right], \quad \frac{1}{\nbo} <\frac{(N-4)\alpha^2 + 3(N-4)\alpha -8}{2N\alpha(\alpha+2)}.
	\end{align}
			
	In the case $\alpha>1$, we take $\frac{1}{\nbo} =0$ or $\nbo=\infty$. 
			
	In the case $\alpha \leq 1$, which together with $\frac{8}{N}<\alpha<\frac{8}{N-4}$ imply $N\geq 8$, we take $\frac{1}{\nbo} = \frac{1-\alpha}{2}$ or $\nbo = \frac{2}{1-\alpha}$. By direct computations, we can check that the above conditions are fulfilled for this choice of $\nbo$. 
			
	Thanks to \eqref{duhamel}, we get from \eqref{est-linear}, \eqref{est-F1} and \eqref{est-F2} that for $\sigma>0$ sufficiently small, there exists $\upsilon = \upsilon(\sigma)>0$ such that
	\[
	\|e^{-i(t-T)(\Delta^2-\mu \Delta)} u(T)\|_{L^{\kbo}([T,\infty), L^{\rbo})} \lesssim \vareps^{\upsilon}.
	\]	
	
	\vspace{2mm}
	
	\noindent {\bf Case 2. $N = 5$.} In this case, the right hand side of the last inequality in \eqref{cond-n} may take negative values in the intercritical regime $\frac{8}{5}<\alpha<8$. We need to proceed differently. Recall that $\kbo=\frac{4\alpha(\alpha+2)}{8-\alpha}$ and  $\rbo=\alpha+2$. We define $(\vartheta,\varrho)$ satisfying
	\[
	\frac{1}{\kbo} = \frac{\theta}{\vartheta}, \quad \frac{1}{\rbo} = \frac{\theta}{\varrho}, \quad \theta = \frac{8}{5\alpha} \in (0,1).
	\]
	It is straightforward to check that $\frac{4}{\vartheta}+\frac{5}{\varrho}=\frac{5}{2}$ and $\varrho =\frac{8(\alpha+2)}{5\alpha}\in (2, 10)$. This shows that $(\vartheta,\varrho) \in B$. We estimate
	\[
	\|F_2\|_{L^{\kbo}([T,\infty), L^{\rbo})} \leq \|F_2\|^{\theta}_{L^{\vartheta}([T,\infty), L^{\varrho})} \|F_2\|^{1-\theta}_{L^\infty([T,\infty),L^\infty)}.
	\]
	Since $(\vartheta,\varrho) \in B$, we have $\|F_2\|_{L^{\vartheta}([T,\infty),L^{\varrho})} \lesssim 1$. Arguing as above, we have
	\begin{align*}
	\|F_2(t)\|_{L^{\infty}} &\lesssim \int_0^{T-\vareps^{-\sigma}} (t-s)^{-\frac{5}{4}} \||u(s)|^\alpha u(s)\|_{L^1} ds \\
	&\lesssim \int_0^{T-\vareps^{-\sigma}} (t-s)^{-\frac{5}{4}} \|u(s)\|^{\alpha+1}_{L^{\alpha+1}} ds \\
	&\lesssim (t-T+\vareps^{-\sigma})^{-\frac{1}{4}}
	\end{align*}
	which implies
	\[
	\|F_2\|_{L^\infty([T,\infty),L^\infty)} \lesssim \vareps^{\frac{\sigma}{4}}.
	\]
	Thus we obtain
	\[
	\|F_2\|_{L^{\kbo}([T,\infty),L^{\rbo})} \lesssim \vareps^{\frac{(5\alpha-8)\sigma}{20\alpha}}.
	\]
	
	%We consider two cases: $N\geq 5$ and $2\leq N\leq 4$.
	
	%\vspace{2mm}
	
	%\noindent {\bf Case 1. $N\geq 5$.} We define $(\vartheta, \varrho)$ satisfying
	%\[
	%\frac{1}{\kbo} = \frac{\theta}{\vartheta}, \quad \frac{1}{\rbo} = \frac{\theta}{\varrho}, \quad \theta =\frac{8}{N\alpha} \in (0,1).
	%\]
	%A direct calculation shows that $\frac{4}{\vartheta} + \frac{N}{\varrho} = \frac{N}{2}$ and
	%\[
	%\varrho = \frac{8(\alpha+2)}{N\alpha}.
	%\]
	%As $\frac{8}{N}<\alpha<\frac{8}{N-4}$, we readily check that $\varrho \in \left(2,\frac{2N}{N-4}\right)$ which shows that $(\vartheta, \varrho) \in B$. By H\"older's inequality, we have
	%\[
	%\|F_2\|_{L^{\kbo}([T,\infty), L^{\rbo})} \leq \|F_2\|^\theta_{L^{\vartheta}([T,\infty), L^{\varrho})} \|F_2\|^{1-\theta}_{L^{\infty}([T,\infty), L^{\infty})}.
	%\]
	%Since $(\vartheta,\varrho) \in B$, we use Strichartz estimates and the fact that
	%\[
	%F_2(t) = e^{-i(t-T+\vareps^{-\sigma})(\Delta^2-\mu \Delta)} u(T-\vareps^{-\sigma}) - e^{-it(\Delta^2-\mu \Delta)} u_0
	%\]
	%to have
	%\[
	%\|F_2\|_{L^{\vartheta}([T, \infty), L^{\varrho})} \lesssim 1.
	%\]
	%On the other hand, by dispersive estimates \eqref{disper-est}, Sobolev embedding and \eqref{est-solu-gwp}, we have for $t\geq T$,
	%\begin{align*}
	%\|F_2(t)\|_{L^{\infty}} &\lesssim \int_0^{T-\vareps^{-\sigma}} (t-s)^{-\frac{N}{4}} \||u(s)|^\alpha u(s)\|_{L^1} ds \\
	%&\lesssim \int_0^{T-\vareps^{-\sigma}} (t-s)^{-\frac{N}{4}} \|u(s)\|^{\alpha+1}_{L^{\alpha+1}} ds \\
	%&\lesssim (t-T+\vareps^{-\sigma})^{-\frac{N-4}{4}}.
	%\end{align*}
	%The problem here is that $\alpha+1$ may be smaller than 2.
	
	\vspace{2mm}

	\noindent {\bf Case 3. $2\leq N \leq 4$.} Recall that we are considering $\alpha>\frac{8}{N}$ here. In this case, the third condition in \eqref{cond-n} does not hold. To overcome the difficulty, we make use of \eqref{mora-est} as follows. By dispersive estimates \eqref{disper-est} and H\"older's inequality, we have for $t\geq T$,
	\begin{align*}
	\|F_2(t)\|_{L^\infty} &\lesssim \int_J (t-s)^{-\frac{N}{4}} \|u(s)\|^{\alpha+1}_{L^{\alpha+1}} ds \\
	&\lesssim \int_J (t-s)^{-\frac{N}{4}} \|u(s)\|^{\frac{(\alpha-1)(\alpha+2)}{\alpha}}_{L^{\alpha+2}} \|u(s)\|^{\frac{2}{\alpha}}_{L^2} ds \\
	&\lesssim \int_J (t-s)^{-\frac{N}{4}} \|u(s)\|^{\frac{(\alpha-1)(\alpha+2)}{\alpha}}_{L^{\alpha+2}} ds \\
	&\lesssim \left\|(t-s)^{-\frac{N}{4}} \right\|_{L^\alpha_s(J)} \| \|u(s)\|^{\frac{(\alpha-1)(\alpha+2)}{\alpha}}_{L^{\alpha+2}} \|_{L^{\frac{\alpha}{\alpha-1}}_s(J)} \\
	&\lesssim \left\|(t-s)^{-\frac{N}{4}} \right\|_{L^\alpha_s(J)} \left(\|u\|^{\alpha+2}_{L^{\alpha+2}(J\times \R^N)} \right)^{\frac{\alpha-1}{\alpha}}.
	\end{align*}
	We see that for $t\geq T$, 
	\begin{align*}
	\left\|(t-s)^{-\frac{N}{4}}\right\|_{L^\alpha_s(J)} & = \left( \int_0^{T-\vareps^{-\sigma}} (t-s)^{-\frac{N\alpha}{4}} ds \right)^{\frac{1}{\alpha}} \\
	&\lesssim (t-T+\vareps^{-\sigma})^{-\frac{N\alpha-4}{4\alpha}},
	\end{align*}
	where we have used the fact that $t\geq t-T+\vareps^{-\sigma}$ as $T>\vareps^{-\sigma}$. On the other hand, by \eqref{mora-est},
	\[
	\|u\|^{\alpha+2}_{L^{\alpha+2}(J\times \R^N)} \lesssim |J|^{\frac{1}{3}} \lesssim T^{\frac{1}{3}}.
	\]
	We infer that for $t\geq T$,
	\[
	\|F_2(t)\|_{L^\infty} \lesssim (t-T+\vareps^{-\sigma})^{-\frac{N\alpha-4}{4\alpha}} T^{\frac{\alpha-1}{3\alpha}}.
	\]
	It follows that
	\begin{align*}
	\|F_2\|_{L^\infty([T,\infty)\times \R^N)} &\lesssim T^{\frac{\alpha-1}{3\alpha}} \vareps^{\frac{(N\alpha-4)\sigma}{4\alpha}}.
	\end{align*}
	We next define
	\[
	a:=\frac{8}{N\alpha}\kbo, \quad b:= \frac{8}{N\alpha} \rbo.
	\]
	It is easy to check that 
	\[
	\frac{4}{a} + \frac{N}{b}=\frac{N}{2}, \quad b\in [2,\infty)
	\]
	which implies that $(a,b)\in B$. By interpolation, we have
	\begin{align} \label{est-F2-low}
	\begin{aligned}
	\|F_2\|_{L^{\kbo}([T,\infty), L^{\rbo})} &\leq \|F_2\|^{\frac{8}{N\alpha}}_{L^a([T,\infty), L^b)}  \|F_2\|^{\frac{N\alpha-8}{N\alpha}}_{L^\infty([T,\infty)\times \R^N)} \\
	&\lesssim \left(T^{\frac{\alpha-1}{3\alpha}} \vareps^{\frac{(N\alpha-4)\sigma}{4\alpha}}\right)^{\frac{N\alpha-8}{N\alpha}}.
	\end{aligned}
	\end{align}
	Collecting \eqref{duhamel}, \eqref{est-linear}, \eqref{est-F1} and \eqref{est-F2-low}, we get
	\[
	\|e^{-i(t-T)(\Delta^2-\mu\Delta)} u(T)\|_{L^{\kbo}([T,\infty), L^{\rbo})} \lesssim \vareps + \vareps^{\frac{2(\alpha+1)(\alpha-\sigma)}{\alpha(\alpha+2)}} + \left(T^{\frac{\alpha-1}{3\alpha}} \vareps^{\frac{(N\alpha-4)\sigma}{4\alpha}}\right)^{\frac{N\alpha-8}{N\alpha}}.
	\]
	By taking $T=\vareps^{-a\sigma}$ with some $a>1$ to be chosen shortly (it ensures $T>\vareps^{-\sigma}$) and choosing $\sigma>0$ small enough, we obtain
	\begin{align} \label{est-low}
	\|e^{-i(t-T)(\Delta^2-\mu\Delta)} u(T)\|_{L^{\kbo}([T,\infty), L^{\rbo})} \lesssim \vareps^{\upsilon}
	\end{align}
	for some $\upsilon>0$. The above estimate requires
	\[
	\frac{N\alpha-4}{4\alpha} - \frac{(\alpha-1)a}{3\alpha} >0 \quad \text{or} \quad a <\frac{3(N\alpha-4)}{4(\alpha-1)}.
	\]
	It now remains to show that
	\[
	\frac{3(N\alpha-4)}{4(\alpha-1)} >1
	\]
	which is satisfied since $\alpha>\frac{8}{N}$ and $2\leq N\leq 4$. This allows us to choose $a>1$ so that \eqref{est-low} holds.

	Combining the above cases, we finish the proof.
\end{proof}

	\noindent {\it Proof of Theorem $\ref{theo-scat}$.}
	It follows immediately from Lemma $\ref{lem-small-gwp}$, Lemma $\ref{lem-small-scat}$ and Proposition $\ref{prop-scat}$.
	\hfill $\Box$

	\section{Finite time blow-up}
	\label{S5}
	\setcounter{equation}{0}
	In this section, we give the proofs of the finite time blow-up given in Theorem $\ref{theo-blup-inter}$ and Theorem $\ref{theo-blup-ener}$. Let us start with the following Morawetz estimates due to Boulenger-Lenzmann \cite{BL}.
	
	\begin{lemma} [Radial Morawetz estimates \cite{BL}] \label{lem-rad-mora-est}
		Let $N\geq 2$, $\mu \geq 0$, $\alpha>0$ and $\alpha \leq \frac{8}{N-4}$ if $N\geq 5$. Let $u\in C([0,T^*), H^2)$ be a radial solution to  the focusing problem \eqref{4NLS}. Let $\varphi_R$ be as in \eqref{def-varphi-R}. Then for any $t\in [0,T^*)$,
		\begin{align*}
		\frac{d}{dt} M_{\varphi_R}(t) \leq 4N\alpha E_\mu(u(t)) &- 2(N\alpha-8) \|\Delta u(t)\|^2_{L^2} - 2(N\alpha-4) \mu \|\nabla u(t)\|^2_{L^2} \\
		&+O \left(R^{-4} + \mu R^{-2} + R^{-2} \|\nabla u(t)\|^2_{L^2} + R^{-\frac{(N-1)\alpha}{2}} \|\nabla u(t)\|^{\frac{\alpha}{2}}_{L^2} \right).
		\end{align*}
	\end{lemma}
	
	We refer the reader to \cite[Lemma 3.1]{BL} for the proof of this result.
	
	\subsection{Finite time blow-up in the mass and energy intercritical case}
	\begin{lemma} \label{lem-K-mu}
		Let $N\geq 1$, $\mu \geq 0$ and $\frac{8}{N}<\alpha<\alpha^*$. Let $u_0 \in H^2$ satisfy \eqref{cond-ener} and \eqref{cond-grad-blup}. Let $u$ be the corresponding solution to  the focusing problem \eqref{4NLS} defined on the maximal forward time interval $[0,T^*)$. Then there exists $\delta= \delta(u_0,Q)>0$ such that for any $t\in [0,T^*)$,
		\begin{align} \label{est-K-mu}
		K_\mu(u(t)) \leq -\delta,
		\end{align}
		where $K_\mu$ is as in \eqref{defi-K-mu}.	
	\end{lemma}
		
	\begin{proof}
		Multiplying $K_\mu(u(t))$ with $[M(u(t))]^{\sigc}$ and using the conservation of mass and energy, we have
		\begin{align*}
		K_\mu(u(t)) [M(u(t))]^{\sigc} &= \left( \frac{N\alpha}{4} E_\mu(u(t)) - \frac{N\alpha-8}{8}\|\Delta u(t)\|^2_{L^2} - \frac{(N\alpha-4)\mu}{8} \|\nabla u(t)\|^2_{L^2} \right) \|u(t)\|^{2\sigc}_{L^2} \\
		&\leq \frac{N\alpha}{4} E_\mu (u(t)) [M(u(t))]^{\sigc} - \frac{N\alpha-8}{8} \left(\|\Delta u(t)\|_{L^2} \|u(t)\|^{\sigc}_{L^2} \right)^2 \\
		&= \frac{N\alpha}{4} E_\mu(u_0) [M(u_0)]^{\sigc} - \frac{N\alpha-8}{8} \left( \|\Delta u(t)\|_{L^2} \|u(t)\|^{\sigc}_{L^2} \right)^2
		\end{align*}
		for all $t\in [0,T^*)$. By \eqref{cond-ener} and \eqref{energy-Q}, there exists $\theta=\theta(u_0,Q)>0$ such that
		\[
		E_\mu(u_0) [M(u_0)]^{\sigc} < (1-\theta) E_0(Q) [M(Q)]^{\sigc} = (1-\theta)\frac{N\alpha-8}{2N\alpha} \left(\|\Delta Q\|_{L^2} \|Q\|^{\sigc}_{L^2}\right)^2.
		\]
		This together with \eqref{est-solu-blup} imply
		\begin{align*}
		K_\mu(u(t))[M(u(t))]^{\sigc} &\leq (1-\theta)\frac{N\alpha-8}{8} \left(\|\Delta Q\|_{L^2} \|Q\|^{\sigc}_{L^2} \right)^2 - \frac{N\alpha-8}{8} \left(\|\Delta Q\|_{L^2} \|Q\|^{\sigc}_{L^2}\right)^2 \\
		&=-\frac{(N\alpha-8)\theta}{8} \|\Delta Q\|^2_{L^2} [M(Q)]^{\sigc}
		\end{align*}
		which shows that
		\[
		K_\mu(u(t)) \leq -\frac{(N\alpha-8)\theta}{8} \|\Delta Q\|^2_{L^2} \left(\frac{M(Q)}{M(u_0)}\right)^{\sigc}=:-\delta
		\]
		for all $t \in [0,T^*)$. The proof is complete.
	\end{proof}
	
	\begin{corollary} \label{coro-K-mu}
		Let $N\geq 1$, $\mu \geq 0$ and $\frac{8}{N}<\alpha<\alpha^*$. Let $u_0 \in H^2$ satisfy \eqref{cond-ener} and \eqref{cond-grad-blup}. Let $u$ be the corresponding solution to  the focusing problem \eqref{4NLS} defined on the maximal forward time interval $[0,T^*)$. Then it holds that
		\begin{align} \label{claim-Delta}
		\inf_{t\in [0,T^*)} \|\Delta u(t)\|_{L^2} \gtrsim 1.
		\end{align}
	\end{corollary}
	
	\begin{proof}
		Assume by contradiction by \eqref{claim-Delta} is not true. Then there exists a time sequence $(t_n)_{n\geq 1} \subset [0,T^*)$ such that $\|\Delta u(t_n)\|_{L^2} \rightarrow 0$ as $n\rightarrow \infty$. By H\"older's inequality, we have
		\[
		\|\nabla u(t_n)\|^2_{L^2} \leq \|u(t_n)\|_{L^2} \|\Delta u(t_n)\|_{L^2} = \|u_0\|_{L^2} \|\Delta u(t_n)\|_{L^2} \rightarrow 0
		\]
		as $n\rightarrow \infty$. Moreover, by the sharp Gagliardo-Nirenberg inequality \eqref{GN-ineq}, 
		\[
		\|u(t_n)\|^{\alpha+2}_{L^{\alpha+2}} \leq C_{\opt} \|\Delta u(t_n)\|^{\frac{N\alpha}{4}}_{L^2} \|u(t_n)\|^{\frac{8-(N-4)\alpha}{4}}_{L^2} = C(u_0) \|\Delta u(t_n)\|^{\frac{N\alpha}{4}}_{L^2} \rightarrow 0
		\]
		as $n\rightarrow \infty$. It follows that
		\[
		K_\mu(u(t_n)) = \|\Delta u(t_n)\|^2_{L^2} + \frac{\mu}{2} \|\nabla u(t_n)\|^2_{L^2} - \frac{N\alpha}{4(\alpha+2)} \|u(t_n)\|^{\alpha+2}_{L^{\alpha+2}} \rightarrow 0
		\]
		as $n\rightarrow \infty$. This contradicts \eqref{est-K-mu}, and the proof is complete.
	\end{proof}

	\begin{lemma} \label{lem-est-mora-rad}
		Let $N\geq 2$, $\mu\geq 0$, $\frac{8}{N} <\alpha<\alpha^*$ and $\alpha \leq 8$. Let $u_0 \in H^2$ be radially symmetric and satisfy \eqref{cond-ener} and \eqref{cond-grad-blup}. Let $u$ be the corresponding solution to  the focusing problem \eqref{4NLS} defined on the maximal forward time interval $[0,T^*)$. Let $\varphi_R$ be as in \eqref{def-varphi-R}. Then there exists $a=a(u_0,Q)>0$ such that
		\begin{align} \label{est-mora-rad}
		\frac{d}{dt} M_{\varphi_R}(t) \leq -a \|\Delta u(t)\|^2_{L^2}
		\end{align}
		for all $t\in [0,T^*)$.
	\end{lemma}
	
	\begin{proof}
		The proof is based on an argument developed in \cite{BCGJ}. Since $u$ is radially symmetric, by Lemma $\ref{lem-rad-mora-est}$, we have for any $R>0$,
		\begin{align*}
		\frac{d}{dt} M_{\varphi_R}(t) \leq 4 N\alpha E_\mu(u(t)) &-2(N\alpha-8) \|\Delta u(t)\|^2_{L^2} -2(N\alpha-4)\mu \|\nabla u(t)\|^2_{L^2} \\
		&+ O \left( R^{-4} +\mu R^{-2} + R^{-2} \|\nabla u(t)\|^2_{L^2} + R^{-\frac{(N-1)\alpha}{2}} \|\nabla u(t)\|^{\frac{\alpha}{2}}_{L^2} \right)
		\end{align*}
		for all $t\in [0,T^*)$. Using the fact that $\|\nabla u(t)\|_{L^2} \leq C(u_0) \|\Delta u(t)\|^{\frac{1}{2}}_{L^2}$, we get
		\begin{align*}
		\frac{d}{dt} M_{\varphi_R} (t) \leq 4 N\alpha E_\mu(u(t)) &-2(N\alpha-8) \|\Delta u(t)\|^2_{L^2} -2(N\alpha-4)\mu \|\nabla u(t)\|^2_{L^2} \\
		&+ O \left( R^{-4} +\mu R^{-2} + R^{-2} \|\Delta u(t)\|_{L^2} + R^{-\frac{(N-1)\alpha}{2}} \|\Delta u(t)\|^{\frac{\alpha}{4}}_{L^2} \right)
		\end{align*}
		for all $t\in [0,T^*)$. By Young's inequality, we have for any $\vareps>0$,
		\[
		R^{-2} \|\Delta u(t)\|_{L^2} \leq \vareps \|\Delta u(t)\|^2_{L^2} + C \vareps^{-1} R^{-4}
		\]
		and for $\alpha<8$,
		\[
		R^{-\frac{(N-1)\alpha}{2}} \|\Delta u(t)\|^{\frac{\alpha}{4}}_{L^2} \leq \vareps \|\Delta u(t)\|^2_{L^2} + C \vareps^{-\frac{\alpha}{8-\alpha}} R^{-\frac{4(N-1)\alpha}{8-\alpha}}.
		\]
		We thus get for any $\vareps>0$ and any $R>0$,
		\begin{align*}
		\frac{d}{dt} M_{\varphi_R}(t) &\leq 4N\alpha E_\mu (u(t)) -2(N\alpha-8) \|\Delta u(t)\|^2_{L^2} - 2(N\alpha-4)\mu\|\nabla u(t)\|^2_{L^2} \\
		&\mathrel{\phantom{\leq}} + \left\{
		\renewcommand*{\arraystretch}{1.4}
		\begin{array}{ccl}
		\left[C\vareps + CR^{-4(N-1)}\right] \|\Delta u(t)\|^2_{L^2} + O \left(R^{-4} + \mu R^{-2} + \vareps^{-1} R^{-4}\right) &\text{if}& \alpha=8, \\
		C\vareps \|\Delta u(t)\|^2_{L^2} + O \left(R^{-4} + \mu R^{-2} + \vareps^{-1} R^{-4} + \vareps^{-\frac{\alpha}{8-\alpha}} R^{-\frac{4(N-1)\alpha}{8-\alpha}} \right) &\text{if}& \alpha<8,
		\end{array}
		\right.
		\end{align*}
		for all $t\in [0,T^*)$, where $C= C(u_0,Q)>0$. 
		
		Let us now fix $t\in [0,T^*)$ and denote
		\[
		\eta:=\frac{4N\alpha |E_\mu (u_0)| +2}{N\alpha-8}.
		\]
		We consider two cases:
		
		\noindent {\bf Case 1.}
		\[
		\|\Delta u(t)\|^2_{L^2} \leq \eta.
		\]
		By Lemma $\ref{lem-K-mu}$, we have
		\[
		4N\alpha E_\mu(u(t)) - 2(N\alpha-8) \|\Delta u(t)\|^2_{L^2} - 2(N\alpha-4) \mu \|\nabla u(t)\|^2_{L^2} =16 K_\mu (u(t)) \leq -16\delta
		\]
		for all $t\in [0,T^*)$. It follows that
		\begin{align*}
		\frac{d}{dt} M_{\varphi_R}(t) \leq -16 \delta + \left\{
		\renewcommand*{\arraystretch}{1.4}
		\begin{array}{ccl}
		\left[C\vareps + CR^{-4(N-1)}\right] \eta + O \left(R^{-4} + \mu R^{-2} + \vareps^{-1} R^{-4}\right) &\text{if}& \alpha=8, \\
		C\vareps \eta + O \left(R^{-4} + \mu R^{-2} + \vareps^{-1} R^{-4} + \vareps^{-\frac{\alpha}{8-\alpha}} R^{-\frac{4(N-1)\alpha}{8-\alpha}} \right) &\text{if}& \alpha<8.
		\end{array}
		\right.
		\end{align*}
		Choosing $\vareps>0$ sufficiently small and $R>0$ sufficiently large, we get
		\[
		\frac{d}{dt}M_{\varphi_R}(t) \leq -8\delta \leq -\frac{8\delta}{\eta} \|\Delta u(t)\|^2_{L^2}.
		\]
		
		\noindent {\bf Case 2.}
		\[
		\|\Delta u(t)\|^2_{L^2} \geq \eta.
		\]
		In this case, we have
		\begin{align*}
		4N\alpha E_\mu(u(t)) -2(N\alpha-8) \|\Delta u(t)\|^2_{L^2} &- 2(N\alpha-4)\mu \|\nabla u(t)\|^2_{L^2} \\
		&\leq 4N\alpha E_\mu(u_0) - (N\alpha-8)\eta - (N\alpha-8)\|\Delta u(t)\|^2_{L^2} \\
		&\leq -2 - (N\alpha-8)\|\Delta u(t)\|^2_{L^2}.
		\end{align*}
		It yields that
		\begin{align*}
		\frac{d}{dt}M_{\varphi_R} (t) \leq -2 &-(N\alpha-8)\|\Delta u(t)\|^2_{L^2} \\
		&+ \left\{
		\renewcommand*{\arraystretch}{1.4}
		\begin{array}{ccl}
		\left[C\vareps + CR^{-4(N-1)}\right] \|\Delta u(t)\|^2_{L^2} + O \left(R^{-4} + \mu R^{-2} + \vareps^{-1} R^{-4}\right) &\text{if}& \alpha=8, \\
		C\vareps \|\Delta u(t)\|^2_{L^2} + O \left(R^{-4} + \mu R^{-2} + \vareps^{-1} R^{-4} + \vareps^{-\frac{\alpha}{8-\alpha}} R^{-\frac{4(N-1)\alpha}{8-\alpha}} \right) &\text{if}& \alpha<8.
		\end{array}
		\right.
		\end{align*}
		If $\alpha=8$, we choose $\vareps>0$ sufficiently small and $R>0$ sufficiently large so that
		\[
		N\alpha-8 - C\vareps - CR^{-4(N-1)} \geq \frac{N\alpha-8}{2}
		\]
		and
		\[
		-2+ O\left(R^{-4} + \mu R^{-2} + \vareps^{-1} R^{-4} \right) \leq 0.
		\]
		If $\alpha<8$, we choose $\vareps>0$ sufficiently small so that
		\[
		N\alpha-8 - C\vareps \geq \frac{N\alpha-8}{2}
		\]
		and then choose $R>0$ sufficiently large depending on $\vareps$ so that
		\[
		-2 + O \left(R^{-4} + \mu R^{-2} + \vareps^{-1} R^{-4} + \vareps^{-\frac{\alpha}{8-\alpha}} R^{-\frac{4(N-1)\alpha}{8-\alpha}} \right) \leq 0.
		\]
		We thus obtain
		\[
		\frac{d}{dt} M_{\varphi_R}(t) \leq -\frac{N\alpha-8}{2} \|\Delta u(t)\|^2_{L^2}.
		\]
		In both cases, the choices of $\vareps$ and $R$ are independent of $t$. We thus prove \eqref{est-mora-rad} with
		\[
		a:= \min \left\{\frac{8\delta}{\eta}, \frac{N\alpha-8}{2}\right\} >0.
		\]
		The proof is complete.
	\end{proof}

	We are now able to prove Theorem $\ref{theo-blup-inter}$. 
	
	\noindent {\it Proof of Theorem $\ref{theo-blup-inter}$.}
	Assume by contradiction that $T^*=\infty$. By \eqref{claim-Delta} and \eqref{est-mora-rad} , we see that
	\[
	\frac{d}{dt} M_{\varphi_R}(t) \leq -C
	\]
	for some $C>0$. Integrating this bound, it yields that $M_{\varphi_R}(t) <0$ for all $t\geq t_0$ with some $t_0 \gg 1$ sufficiently large. Taking the integration over $[t_0,t]$ of \eqref{est-mora-rad}, we get
	\[
	M_{\varphi_R}(t) \leq -a \int_{t_0}^t \|\Delta u(s)\|^2_{L^2} ds
	\]
	for all $t\geq t_0$. On the other hand, by H\"older's inequality and the conservation of mass, 
	\[
	|M_{\varphi_R}(t)| \leq \|\nabla \varphi_R\|_{L^\infty} \|u(t)\|_{L^2} \|\nabla u(t)\|_{L^2} \leq C(u_0,R) \|\Delta u(t)\|^{\frac{1}{2}}_{L^2}.
	\]
	We infer that
	\begin{align} \label{est-M-varphi-R}
	M_{\varphi_R}(t) \leq -A \int_{t_0}^t |M_{\varphi_R}(s)|^4 ds
	\end{align}
	for some constant $A=A(a, u_0,R)>0$. Set
	\[
	z(t):= \int_{t_0}^t |M_{\varphi_R}(s)|^4 ds, \quad t\geq t_0
	\]
	and fix some $t_1>t_0$. We see that $z(t)$ is strictly increasing, non-negative and satisfies 
	\[
	z'(t) = |M_{\varphi_R}(t)|^4 \geq A^4 [z(t)]^4.
	\]
	Integrating the above inequality over $[t_1,t]$, we get
	\[
	z(t) \geq \frac{z(t_1)}{[1-3A^4[z(t_1)]^3(t-t_1)]^{\frac{1}{3}}}
	\]
	for all $t\geq t_1$. It follows that
	\[
	z(t) \rightarrow \infty \text{ as } t \nearrow t^*:= t_1+\frac{1}{3A^4[z(t_1)]^3}.
	\] 
	By \eqref{est-M-varphi-R}, 
	\[
	M_{\varphi_R}(t) \leq -A z(t) \rightarrow -\infty \text{ as } t\nearrow t^*.
	\]
	Therefore the solution cannot exist for all time $t\geq 0$, and consequencely, we must have $T^*<\infty$. The proof is complete.
	\hfill $\Box$
	
	\subsection{Finite time blow-up in the energy critical case}
	In this subsection, we give the proof of the finite time blow-up given in Theorem $\ref{theo-blup-ener}$. Instead of using the sharp Gagliardo-Nirenberg inequality, we make use of the sharp Sobolev embedding
	\begin{align} \label{SE-ineq}
	\|f\|_{L^{\frac{2N}{N-4}}} \leq C_{\opt} \|\Delta f\|_{L^2}.
	\end{align}
	It is known (see \cite{BL}) that the optimal constant is attained by $W$, i.e.
	\[
	C_{\opt} = \|W\|_{L^{\frac{2N}{N-4}}} \div \|\Delta W\|_{L^2},
	\]
	where $W$ is the unique radial non-negative solution to \eqref{ell-equ-W}. We also have the following identities (see \cite[Appendix]{BL}):
	\begin{align}
	\|\Delta W\|^2_{L^2} &= \|W\|^{\frac{2N}{N-4}}_{L^{\frac{2N}{N-4}}}, \nonumber\\
	E_0(W) &= \frac{2}{N} \|\Delta W\|^2_{L^2}. \label{energy-W}
	\end{align}
	In particular,
	\begin{align} \label{opt-cons-SE}
	C_{\opt} = \|\Delta W\|^{-\frac{4}{N}}_{L^2} = \|W\|_{L^{\frac{2N}{N-4}}}^{-\frac{4}{N-4}} = \left[\frac{N}{2} E_0(W) \right]^{-\frac{2}{N}}.
	\end{align}
	
	\begin{lemma} \label{lem-est-solu-ener}
		Let $N\geq 5$, $\mu\geq 0$ and $\alpha=\frac{8}{N-4}$. Let $u_0 \in H^2$ satisfy \eqref{cond-ener-ener} and \eqref{cond-grad-ener}. Then the corresponding solution to  the focusing problem \eqref{4NLS} satisfies
		\begin{align} \label{est-solu-ener}
		\|\Delta u(t)\|_{L^2}> \|\Delta W\|_{L^2}
		\end{align}
		for all $t$ in the existence time.
	\end{lemma}

	\begin{proof}
		By the sharp Sobolev embedding \eqref{SE-ineq}, we have
		\begin{align*}
		E_\mu(u(t))&=\frac{1}{2} \|\Delta u(t)\|^2_{L^2} +\frac{\mu}{2} \|\nabla u(t)\|^2_{L^2} -\frac{N-4}{2N}\|u(t)\|^{\frac{2N}{N-4}}_{L^{\frac{2N}{N-4}}} \\
		&\geq \frac{1}{2} \|\Delta u(t)\|^2_{L^2} - \frac{N-4}{2N}[C_{\opt}]^{\frac{2N}{N-4}} \|\Delta u(t)\|^{\frac{2N}{N-4}}_{L^2} \\
		&= g(\|\Delta u(t)\|_{L^2}),
		\end{align*}
		where
		\[
		g(\lambda):=\frac{1}{2}\lambda^2- \frac{N-4}{2N}[C_{\opt}]^{\frac{2N}{N-4}} \lambda^{\frac{2N}{N-4}}.
		\]
		By \eqref{opt-cons-SE}, we see that
		\[
		g(\|\Delta W\|_{L^2}) = \frac{2}{N} \|\Delta W\|^2_{L^2} = E_0(W).
		\]
		Thanks to the conservation of energy and \eqref{cond-ener-ener}, we get
		\[
		g(\|\Delta u(t)\|_{L^2}) \leq E_\mu(u(t)) = E_\mu(u_0) <E_0(W) = g(\|\Delta W\|_{L^2})
		\]
		for all $t$ in the existence time. By \eqref{cond-grad-ener}, the continuity argument yields
		\[
		\|\Delta u(t)\|_{L^2} >\|\Delta W\|_{L^2}
		\]
		for all $t$ in the existence time. 
	\end{proof}

	\begin{lemma} \label{lem-K-mu-ener}
		Let $N\geq 5$, $\mu \geq 0$ and $\alpha=\frac{8}{N-4}$. Let $u_0 \in H^2$ satisfy \eqref{cond-ener-ener} and \eqref{cond-grad-ener}. Let $u$ be the corresponding solution to  the focusing problem \eqref{4NLS} defined on the maximal forward time interval $[0,T^*)$. Then there exists $\delta= \delta(u_0,W)>0$ such that for any $t\in [0,T^*)$,
		\begin{align} \label{est-K-mu-ener}
		K_\mu(u(t)) \leq -\delta,
		\end{align}
		where $K_\mu$ is as in \eqref{defi-K-mu}.	
	\end{lemma}
	
	\begin{proof}
		We have
		\begin{align*}
		K_\mu(u(t)) &= \frac{2N}{N-4} E_\mu(u(t)) - \frac{4}{N-4}\|\Delta u(t)\|^2_{L^2} - \frac{(N+4)\mu}{2(N-4)} \|\nabla u(t)\|^2_{L^2} \\
		&\leq \frac{2N}{N-4} E_\mu (u(t))  - \frac{4}{N-4} \|\Delta u(t)\|^2_{L^2} 
		\end{align*}
		for all $t\in [0,T^*)$. By \eqref{cond-ener-ener} and \eqref{energy-W}, there exists $\theta=\theta(u_0,W)>0$ such that
		\[
		E_\mu(u_0) < (1-\theta) E_0(Q) = (1-\theta) \frac{2}{N}\|\Delta W\|^2_{L^2}.
		\]
		This together with \eqref{est-solu-ener} imply that
		\begin{align*}
		K_\mu(u(t)) \leq (1-\theta)\frac{4}{N-4} \|\Delta W\|^2_{L^2} - \frac{4}{N-4} \|\Delta W\|^2_{L^2} = - \frac{4\theta}{N-4} \|\Delta W\|^2_{L^2}=:-\delta
		\end{align*}
		for all $t \in [0,T^*)$. The proof is complete.
	\end{proof}
	
	\begin{lemma} \label{lem-est-mora-rad-ener}
		Let $N\geq 5$, $\mu\geq 0$ and $\alpha=\frac{8}{N-4}$. Let $u_0 \in H^2$ be radially symmetric and satisfy \eqref{cond-ener-ener} and \eqref{cond-grad-ener}. Let $u$ be the corresponding solution to  the focusing problem \eqref{4NLS} defined on the maximal forward time interval $[0,T^*)$. Let $\varphi_R$ be as in \eqref{def-varphi-R}. Then there exists $a=a(u_0,W)>0$ such that
		\begin{align} \label{est-mora-rad-ener}
		\frac{d}{dt} M_{\varphi_R}(t) \leq -a \|\Delta u(t)\|^2_{L^2}
		\end{align}
		for all $t\in [0,T^*)$.
	\end{lemma}
	
	\begin{proof}
		The proof is similar to that of Lemm $\ref{lem-est-mora-rad}$. For the reader's convenience, we give some details. Since $u$ is radially symmetric, we apply Lemma $\ref{lem-rad-mora-est}$ to have for any $R>0$,
		\begin{align*}
		\frac{d}{dt} M_{\varphi_R}(t) \leq \frac{32N}{N-4} E_\mu(u(t)) &-\frac{64}{N-4} \|\Delta u(t)\|^2_{L^2} -\frac{8(N+4)\mu}{N-4} \|\nabla u(t)\|^2_{L^2} \\
		&+ O \left( R^{-4} +\mu R^{-2} + R^{-2} \|\nabla u(t)\|^2_{L^2} + R^{-\frac{4(N-1)}{N-4}} \|\nabla u(t)\|^{\frac{4}{N-4}}_{L^2} \right)
		\end{align*}
		for all $t\in [0,T^*)$. Using the fact that $\|\nabla u(t)\|_{L^2} \leq C(u_0) \|\Delta u(t)\|^{\frac{1}{2}}_{L^2}$, we get
		\begin{align*}
		\frac{d}{dt} M_{\varphi_R} (t) \leq \frac{32N}{N-4} E_\mu(u(t)) &-\frac{64}{N-4} \|\Delta u(t)\|^2_{L^2} -\frac{8(N+4)\mu}{N-4} \|\nabla u(t)\|^2_{L^2} \\
		&+ O \left( R^{-4} +\mu R^{-2} + R^{-2} \|\Delta u(t)\|_{L^2} + R^{-\frac{4(N-1)}{N-4}} \|\Delta u(t)\|^{\frac{2}{N-4}}_{L^2} \right)
		\end{align*}
		for all $t\in [0,T^*)$. By the Young's inequality, we get for any $\vareps>0$ and any $R>0$,
		\begin{align*}
		\frac{d}{dt} M_{\varphi_R}(t) &\leq \frac{32N}{N-4} E_\mu (u(t)) -\frac{64}{N-4} \|\Delta u(t)\|^2_{L^2} - \frac{8(N+4)\mu}{N-4}\|\nabla u(t)\|^2_{L^2} \\
		&\mathrel{\phantom{\leq}} + \left\{
		\renewcommand*{\arraystretch}{1.4}
		\begin{array}{ccl}
		\left[C\vareps + CR^{-16}\right] \|\Delta u(t)\|^2_{L^2} + O \left(R^{-4} + \mu R^{-2} + \vareps^{-1} R^{-4}\right) &\text{if}& N=5, \\
		C\vareps \|\Delta u(t)\|^2_{L^2} + O \left(R^{-4} + \mu R^{-2} + \vareps^{-1} R^{-4} + \vareps^{-\frac{1}{N-5}} R^{-\frac{4(N-1)}{N-5}} \right) &\text{if}& N\geq 6,
		\end{array}
		\right.
		\end{align*}
		for all $t\in [0,T^*)$, where $C= C(u_0,W)>0$. 
		
		Let us now fix $t\in [0,T^*)$ and denote
		\[
		\eta:=N|E_\mu(u_0)| + \frac{N-4}{16}.
		\]
		We consider two cases:
		
		\noindent {\bf Case 1.}
		\[
		\|\Delta u(t)\|^2_{L^2} \leq \eta.
		\]
		By \eqref{est-K-mu-ener}, we have
		\[
		\frac{32N}{N-4} E_\mu(u(t)) - \frac{64}{N-4} \|\Delta u(t)\|^2_{L^2} - \frac{8(N+4)\mu}{N-4}  \|\nabla u(t)\|^2_{L^2} =16 K_\mu (u(t)) \leq -16\delta
		\]
		for all $t\in [0,T^*)$. It follows that
		\begin{align*}
		\frac{d}{dt} M_{\varphi_R}(t) \leq -16 \delta + \left\{
		\renewcommand*{\arraystretch}{1.4}
		\begin{array}{ccl}
		\left[C\vareps + CR^{-16}\right] \eta + O \left(R^{-4} + \mu R^{-2} + \vareps^{-1} R^{-4}\right) &\text{if}& N=5, \\
		C\vareps \eta + O \left(R^{-4} + \mu R^{-2} + \vareps^{-1} R^{-4} + \vareps^{-\frac{1}{N-5}} R^{-\frac{4(N-1)}{N-5}} \right) &\text{if}& N\geq 6.
		\end{array}
		\right.
		\end{align*}
		By choosing $\vareps>0$ sufficiently small and $R>0$ sufficiently large, we get
		\[
		\frac{d}{dt}M_{\varphi_R}(t) \leq -8\delta \leq -\frac{8\delta}{\eta} \|\Delta u(t)\|^2_{L^2}.
		\]
		
		\noindent {\bf Case 2.}
		\[
		\|\Delta u(t)\|^2_{L^2} \geq \eta.
		\]
		In this case, we have
		\begin{align*}
		\frac{32N}{N-4} E_\mu(u(t)) -\frac{64}{N-4} \|\Delta u(t)\|^2_{L^2} &- \frac{8(N+4)\mu}{N-4} \|\nabla u(t)\|^2_{L^2} \\
		&\leq \frac{32N}{N-4} E_\mu(u_0) - \frac{32}{N-4}\eta - \frac{32}{N-4}\|\Delta u(t)\|^2_{L^2} \\
		&\leq -2 - \frac{32}{N-4}\|\Delta u(t)\|^2_{L^2}.
		\end{align*}
		It yields that
		\begin{align*}
		\frac{d}{dt}M_{\varphi_R} (t) \leq -2 &-\frac{32}{N-4}\|\Delta u(t)\|^2_{L^2} \\
		&+ \left\{
		\renewcommand*{\arraystretch}{1.4}
		\begin{array}{ccl}
		\left[C\vareps + CR^{-16}\right] \|\Delta u(t)\|^2_{L^2} + O \left(R^{-4} + \mu R^{-2} + \vareps^{-1} R^{-4}\right) &\text{if}& N=5, \\
		C\vareps \|\Delta u(t)\|^2_{L^2} + O \left(R^{-4} + \mu R^{-2} + \vareps^{-1} R^{-4} + \vareps^{-\frac{1}{N-5}} R^{-\frac{4(N-1)}{N-5}} \right) &\text{if}& N\geq 6.
		\end{array}
		\right.
		\end{align*}
		If $N=5$, we choose $\vareps>0$ sufficiently small and $R>0$ sufficiently large so that
		\[
		32 - C\vareps - CR^{-16} \geq 16
		\]
		and
		\[
		-2+ O\left(R^{-4} + \mu R^{-2} + \vareps^{-1} R^{-4} \right) \leq 0.
		\]
		If $N\geq 6$, we choose $\vareps>0$ sufficiently small so that
		\[
		\frac{32}{N-4} - C\vareps \geq \frac{16}{N-4}
		\]
		and then choose $R>0$ sufficiently large depending on $\vareps$ so that
		\[
		-2 + O \left(R^{-4} + \mu R^{-2} + \vareps^{-1} R^{-4} + \vareps^{-\frac{1}{N-5}} R^{-\frac{4(N-1)}{N-5}} \right) \leq 0.
		\]
		We thus obtain
		\[
		\frac{d}{dt} M_{\varphi_R}(t) \leq -\frac{16}{N-4} \|\Delta u(t)\|^2_{L^2}.
		\]
		In both cases, the choices of $\vareps$ and $R$ are independent of $t$. We thus prove \eqref{est-mora-rad-ener} with
		\[
		a:= \min \left\{\frac{8\delta}{\eta}, \frac{16}{N-4}\right\} >0.
		\]
		The proof is complete.
	\end{proof}
	
	We are now able to prove Theorem $\ref{theo-blup-ener}$. 
	
	\noindent {\it Proof of Theorem $\ref{theo-blup-ener}$.}
	The proof is completely similar to that of Theorem $\ref{theo-blup-inter}$ using \eqref{est-solu-ener} and \eqref{est-mora-rad-ener}. We thus omit the details.
	\hfill $\Box$

	\section*{Acknowledgement}
	This work was supported in part by the Labex CEMPI (ANR-11-LABX-0007-01). V. D. D. would like to express his deep gratitude to his wife - Uyen Cong for her encouragement and support. %He also would like to thank the reviewer for his/her helpful comments and suggestions. 


\begin{bibdiv}
		\begin{biblist}
			
			\bib{BF}{article}{
				author={Baruch, G.},
				author={Fibich, G.},
				title={Singular solutions of the $L^2$-supercritical biharmonic nonlinear
					Schr\"{o}dinger equation},
				journal={Nonlinearity},
				volume={24},
				date={2011},
				number={6},
				pages={1843--1859},
				issn={0951-7715},
				%review={\MR{2802308}},
				%doi={10.1088/0951-7715/24/6/009},
			}		
			
			\bib{BFM}{article}{
				author={Baruch, G.},
				author={Fibich, G.},
				author={Mandelbaum, E.},
				title={Ring-type singular solutions of the biharmonic nonlinear
					Schr\"{o}dinger equation},
				journal={Nonlinearity},
				volume={23},
				date={2010},
				number={11},
				pages={2867--2887},
				issn={0951-7715},
				%review={\MR{2727174}},
				%doi={10.1088/0951-7715/23/11/008},
			}		
			
			\bib{BFM-SIAM}{article}{
				author={Baruch , G.},
				author={Fibich, G.},
				author={Mandelbaum, E.},
				title={Singular solutions of the biharmonic nonlinear Schr\"{o}dinger
					equation},
				journal={SIAM J. Appl. Math.},
				volume={70},
				date={2010},
				number={8},
				pages={3319--3341},
				issn={0036-1399},
				%review={\MR{2763506}},
				%doi={10.1137/100784199},
			}
			
			\bib{BaKS}{article}{
				author={Ben-Artzi, M.},
				author={Koch, H.},
				author={Saut, J. C.},
				title={Dispersion estimates for fourth order Schr\"{o}dinger equations},
				language={English, with English and French summaries},
				journal={C. R. Acad. Sci. Paris S\'{e}r. I Math.},
				volume={330},
				date={2000},
				number={2},
				pages={87--92},
				issn={0764-4442},
				%review={\MR{1745182}},
				%doi={10.1016/S0764-4442(00)00120-8},
			}
			
			\bib{BCGJ}{article}{
				author={Bonheure, D.},
				author={Cast\'{e}ras, J. B.},
				author={Gou, T.},
				author={Jeanjean, L.},
				title={Strong instability of ground states to a fourth order Schr\"{o}dinger
					equation},
				journal={Int. Math. Res. Not. IMRN},
				date={2019},
				number={17},
				pages={5299--5315},
				issn={1073-7928},
				%review={\MR{4001029}},
				%doi={10.1093/imrn/rnx273},
			}
			
			\bib{BL}{article}{
				author={Boulenger, T.},
				author={Lenzmann, E.},
				title={Blowup for biharmonic NLS},
				language={English, with English and French summaries},
				journal={Ann. Sci. \'{E}c. Norm. Sup\'{e}r. (4)},
				volume={50},
				date={2017},
				number={3},
				pages={503--544},
				issn={0012-9593},
				%review={\MR{3665549}},
				%doi={10.24033/asens.2326},
			}
			
			
			\bib{Cazenave}{book}{
				author={Cazenave, T.},
				title={Semilinear Schr\"{o}dinger equations},
				series={Courant Lecture Notes in Mathematics},
				volume={10},
				publisher={New York University, Courant Institute of Mathematical
					Sciences, New York; American Mathematical Society, Providence, RI},
				date={2003},
				pages={xiv+323},
				isbn={0-8218-3399-5},
				%review={\MR{2002047}},
				%doi={10.1090/cln/010},
			}	
			
			\bib{Dinh-BBMS}{article}{
				author={Dinh, V. D.},
				title={On well-posedness, regularity and ill-posedness for the nonlinear
					fourth-order Schr\"{o}dinger equation},
				journal={Bull. Belg. Math. Soc. Simon Stevin},
				volume={25},
				date={2018},
				number={3},
				pages={415--437},
				%issn={1370-1444},
				%review={\MR{3852677}},
			}
			
			\bib{Dinh-DPDE}{article}{
				author={Dinh , V. D.},
				title={On the focusing mass-critical nonlinear fourth-order Schr\"{o}dinger
					equation below the energy space},
				journal={Dyn. Partial Differ. Equ.},
				volume={14},
				date={2017},
				number={3},
				pages={295--320},
				issn={1548-159X},
				%review={\MR{3702543}},
				%doi={10.4310/DPDE.2017.v14.n3.a4},
			}
			
			\bib{Dinh-NA}{article}{
				author={Dinh, V.  D.},
				title={Global existence and scattering for a class of nonlinear
					fourth-order Schr\"{o}dinger equation below the energy space},
				journal={Nonlinear Anal.},
				volume={172},
				date={2018},
				pages={115--140},
				issn={0362-546X},
				%review={\MR{3790370}},
				%doi={10.1016/j.na.2018.03.003},
			}
			
			\bib{Dinh-JDDE}{article}{
				author={V. D. Dinh},
				title={On blowup solutions to the focusing intercritical nonlinear
					fourth-order Schr\"{o}dinger equation},
				journal={J. Dynam. Differential Equations},
				volume={31},
				date={2019},
				number={4},
				pages={1793--1823},
				issn={1040-7294},
				%review={\MR{4028554}},
				%doi={10.1007/s10884-018-9690-y},
			}
			
			\bib{DK}{article}{
				author={Dinh, V.  D.},
				author={Keraani, S.},
				title={The Sobolev-Morawetz approach for the energy scattering of nonlinear Schr\"{o}dinger-type equations with radial data},
				journal={Discrete Contin. Dyn. Syst. Ser. S (in press)},
				%volume={31},
				%date={2019},
				%number={4},
				%pages={1793--1823},
				%issn={1040-7294},
				%review={\MR{4028554}},
				%doi={http://dx.doi.org/10.3934/dcdss.2020407},
			}
			
			\bib{DM}{article}{
				author={Dodson, B.},
				author={Murphy, J.},
				title={A new proof of scattering below the ground state for the 3D radial
					focusing cubic NLS},
				journal={Proc. Amer. Math. Soc.},
				volume={145},
				date={2017},
				number={11},
				pages={4859--4867},
				issn={0002-9939},
				%review={\MR{3692001}},
				%doi={10.1090/proc/13678},
			}
			
			\bib{FIP}{article}{
				author={Fibich, G.},
				author={Ilan, B.},
				author={Papanicolaou, G.},
				title={Self-focusing with fourth-order dispersion},
				journal={SIAM J. Appl. Math.},
				volume={62},
				date={2002},
				number={4},
				pages={1437--1462},
				issn={0036-1399},
				%review={\MR{1898529}},
				%doi={10.1137/S0036139901387241},
			}
			
			\bib{Foschi}{article}{
				author={Foschi, D.},
				title={Inhomogeneous Strichartz estimates},
				journal={J. Hyperbolic Differ. Equ.},
				volume={2},
				date={2005},
				number={1},
				pages={1--24},
				issn={0219-8916},
				%review={\MR{2134950}},
				%doi={10.1142/S0219891605000361},
			}
			
			\bib{Guo}{article}{
				author={Guo, Q.},
				title={Scattering for the focusing $L^2$-supercritical and
					$\dot{H}^2$-subcritical biharmonic NLS equations},
				journal={Comm. Partial Differential Equations},
				volume={41},
				date={2016},
				number={2},
				pages={185--207},
				issn={0360-5302},
				%review={\MR{3462127}},
				%doi={10.1080/03605302.2015.1116556},
			}
			
			\bib{HK}{article}{
				author={Hmidi, T.},
				author={Keraani, S.},
				title={Blowup theory for the critical nonlinear Schr\"{o}dinger equations
					revisited},
				journal={Int. Math. Res. Not.},
				date={2005},
				number={46},
				pages={2815--2828},
				issn={1073-7928},
				%review={\MR{2180464}},
				%doi={10.1155/IMRN.2005.2815},
			}
			
			\bib{HR}{article}{
				author={Holmer, J.},
				author={Roudenko, S.},
				title={A sharp condition for scattering of the radial 3D cubic nonlinear
					Schr\"{o}dinger equation},
				journal={Comm. Math. Phys.},
				volume={282},
				date={2008},
				number={2},
				pages={435--467},
				issn={0010-3616},
				%review={\MR{2421484}},
				%doi={10.1007/s00220-008-0529-y},
			}
			
			\bib{Kato}{article}{
				author={Kato, T.},
				title={On nonlinear Schr\"{o}dinger equations. II. $H^s$-solutions and
					unconditional well-posedness},
				journal={J. Anal. Math.},
				volume={67},
				date={1995},
				pages={281--306},
				issn={0021-7670},
				%review={\MR{1383498}},
				%doi={10.1007/BF02787794},
			}
			
			\bib{KT}{article}{
				author={Keel, M.},
				author={Tao, T.},
				title={Endpoint Strichartz estimates},
				journal={Amer. J. Math.},
				volume={120},
				date={1998},
				number={5},
				pages={955--980},
				issn={0002-9327},
				%review={\MR{1646048}},
			}
			
			\bib{Karpman}{article}{
				title = {Stabilization of soliton instabilities by higher-order dispersion: Fourth-order nonlinear Schr\"{o}dinger-type equations},
				author = {Karpman, V. I.},
				journal = {Phys. Rev. E},
				volume = {53},
				issue = {2},
				pages = {R1336--R1339},
				numpages = {0},
				year = {1996},
				month = {Feb},
				publisher = {American Physical Society},
				%doi = {10.1103/PhysRevE.53.R1336},
				%url = {https://link.aps.org/doi/10.1103/PhysRevE.53.R1336}
			}
			
			\bib{KS}{article}{
				author={Karpman, V. I.},
				author={Shagalov, A. G.},
				title={Solitons and their stability in high dispersive systems. I.
					Fourth-order nonlinear Schr\"{o}dinger-type equations with power-law
					nonlinearities},
				journal={Phys. Lett. A},
				volume={228},
				date={1997},
				number={1-2},
				pages={59--65},
				issn={0375-9601},
				%review={\MR{1448598}},
				%doi={10.1016/S0375-9601(97)00063-7},
			}
			
			\bib{KM}{article}{
				author={Kenig, C. E.},
				author={Merle, F.},
				title={Global well-posedness, scattering and blow-up for the
					energy-critical, focusing, non-linear Schr\"{o}dinger equation in the radial
					case},
				journal={Invent. Math.},
				volume={166},
				date={2006},
				number={3},
				pages={645--675},
				issn={0020-9910},
				%review={\MR{2257393}},
				%doi={10.1007/s00222-006-0011-4},
			}
			
			\bib{Lions}{article}{
				author={Lions, P. L.},
				title={The concentration-compactness principle in the calculus of
					variations. The locally compact case. I},
				language={English, with French summary},
				journal={Ann. Inst. H. Poincar\'{e} Anal. Non Lin\'{e}aire},
				volume={1},
				date={1984},
				number={2},
				pages={109--145},
				issn={0294-1449},
				%review={\MR{778970}},
			}
			
			\bib{MXZ-1}{article}{
				author={Miao, C.},
				author={Xu, G.},
				author={Zhao, L.},
				title={Global well-posedness and scattering for the focusing
					energy-critical nonlinear Schr\"{o}dinger equations of fourth order in the
					radial case},
				journal={J. Differential Equations},
				volume={246},
				date={2009},
				number={9},
				pages={3715--3749},
				issn={0022-0396},
				%review={\MR{2515176}},
				%doi={10.1016/j.jde.2008.11.011},
			}
			
		
			\bib{MXZ-2}{article}{
				author={Miao , C.},
				author={Xu, G.},
				author={Zhao, L.},
				title={Global well-posedness and scattering for the defocusing
					energy-critical nonlinear Schr\"{o}dinger equations of fourth order in
					dimensions $d\geqslant9$},
				journal={J. Differential Equations},
				volume={251},
				date={2011},
				number={12},
				pages={3381--3402},
				issn={0022-0396},
				%review={\MR{2837688}},
				%doi={10.1016/j.jde.2011.08.009},
			}
			
			\bib{MWZ}{article}{
				author={Miao, C.},
				author={Wu, H.},
				author={Zhang, J.},
				title={Scattering theory below energy for the cubic fourth-order
					Schr\"{o}dinger equation},
				journal={Math. Nachr.},
				volume={288},
				date={2015},
				number={7},
				pages={798--823},
				issn={0025-584X},
				%review={\MR{3345105}},
				%doi={10.1002/mana.201400012},
			}
			
			
			\bib{MZ}{article}{
				author={Miao, C.},
				author={Zheng, J.},
				title={Scattering theory for the defocusing fourth-order Schr\"{o}dinger
					equation},
				journal={Nonlinearity},
				volume={29},
				date={2016},
				number={2},
				pages={692--736},
				issn={0951-7715},
				%review={\MR{3461612}},
				%doi={10.1088/0951-7715/29/2/692},
			}
			
			\bib{Pausader-DPDE}{article}{
				author={Pausader, B.},
				title={Global well-posedness for energy critical fourth-order Schr\"{o}dinger
					equations in the radial case},
				journal={Dyn. Partial Differ. Equ.},
				volume={4},
				date={2007},
				number={3},
				pages={197--225},
				issn={1548-159X},
				%review={\MR{2353631}},
				%doi={10.4310/DPDE.2007.v4.n3.a1},
			}
			
			\bib{Pausader-DCDS}{article}{
				author={Pausader , B.},
				title={The focusing energy-critical fourth-order Schr\"{o}dinger equation
					with radial data},
				journal={Discrete Contin. Dyn. Syst.},
				volume={24},
				date={2009},
				number={4},
				pages={1275--1292},
				issn={1078-0947},
				%review={\MR{2505703}},
				%doi={10.3934/dcds.2009.24.1275},
			}
			
			\bib{Pausader-JFA}{article}{
				author={Pausader, B. },
				title={The cubic fourth-order Schr\"{o}dinger equation},
				journal={J. Funct. Anal.},
				volume={256},
				date={2009},
				number={8},
				pages={2473--2517},
				issn={0022-1236},
				%review={\MR{2502523}},
				%doi={10.1016/j.jfa.2008.11.009},
			}
			
			\bib{PS}{article}{
				author={Pausader, B.},
				author={Shao, S.},
				title={The mass-critical fourth-order Schr\"{o}dinger equation in high
					dimensions},
				journal={J. Hyperbolic Differ. Equ.},
				volume={7},
				date={2010},
				number={4},
				pages={651--705},
				issn={0219-8916},
				%review={\MR{2746203}},
				%doi={10.1142/S0219891610002256},
			}
			
			\bib{PX}{article}{
				author={Pausader, B.},
				author={Xia, S.},
				title={Scattering theory for the fourth-order Schr\"{o}dinger equation in low
					dimensions},
				journal={Nonlinearity},
				volume={26},
				date={2013},
				number={8},
				pages={2175--2191},
				issn={0951-7715},
				%review={\MR{3078112}},
				%doi={10.1088/0951-7715/26/8/2175},
			}
			
			\bib{Strauss}{article}{
				author={Strauss, W. A.},
				title={Existence of solitary waves in higher dimensions},
				journal={Comm. Math. Phys.},
				volume={55},
				date={1977},
				number={2},
				pages={149--162},
				issn={0010-3616},
				%review={\MR{454365}},
			}
			
			\bib{Vilela}{article}{
				author={Vilela, M. C.},
				title={Inhomogeneous Strichartz estimates for the Schr\"{o}dinger equation},
				journal={Trans. Amer. Math. Soc.},
				volume={359},
				date={2007},
				number={5},
				pages={2123--2136},
				issn={0002-9947},
				%review={\MR{2276614}},
				%doi={10.1090/S0002-9947-06-04099-2},
			}
			
			\bib{ZYZ-DPDE}{article}{
				author={Zhu, S.},
				author={Zhang, J.},
				author={Yang, H.},
				title={Limiting profile of the blow-up solutions for the fourth-order
					nonlinear Schr\"{o}dinger equation},
				journal={Dyn. Partial Differ. Equ.},
				volume={7},
				date={2010},
				number={2},
				pages={187--205},
				issn={1548-159X},
				%review={\MR{2675546}},
				%doi={10.4310/DPDE.2010.v7.n2.a4},
			}
			
			\bib{ZYZ-NA}{article}{
				author={Zhu, S.},
				author={Yang, H.},
				author={Zhang, J.},
				title={Blow-up of rough solutions to the fourth-order nonlinear
					Schr\"{o}dinger equation},
				journal={Nonlinear Anal.},
				volume={74},
				date={2011},
				number={17},
				pages={6186--6201},
				issn={0362-546X},
				%review={\MR{2833404}},
				%doi={10.1016/j.na.2011.05.096},
			}
			
		\end{biblist}
	\end{bibdiv}
\end{document}